\documentclass[11pt]{article}
\usepackage{amsmath}
\usepackage{amscd}
\usepackage{amsthm}
\usepackage{verbatim}
\usepackage{fullpage}
\usepackage{amssymb}
\usepackage{amsfonts}
\usepackage{mathrsfs}
\usepackage{enumitem}
\setlist{itemsep=1pt}

\usepackage{stmaryrd}
\usepackage[usenames]{color}
\usepackage{graphicx}
\usepackage{hyperref}
\usepackage{mathscinet}
\usepackage{libertine}
\usepackage[libertine]{newtxmath}
\usepackage[normalem]{ulem}
\usepackage[all]{xy}

\usepackage[backend=biber,style=alphabetic,doi=false,isbn=false,url=false,minnames=6,maxnames=6]{biblatex}
\renewbibmacro{in:}{
  \ifentrytype{article}{}{\printtext{\bibstring{in}\intitlepunct}}}
\renewbibmacro*{volume+number+eid}{
  \printtext{vol.}
  \printfield{volume}
  \setunit*{\addnbspace}
  \printfield{number}
  \setunit{\addcomma\space}
  \printfield{eid}}
\DeclareFieldFormat[article]{number}{\mkbibparens{#1}}

\newcommand{\doi}[1]{\href{https://dx.doi.org/#1}{{\small DOI:#1}}}

\newcommand{\googlebooks}[1]{(preview at \href{https://books.google.com/books?id=#1}{google books})}

\newcommand{\numdam}[1]{}

\theoremstyle{plain}
\newtheorem{thm}{Theorem}[section]
\newtheorem{theorem}[thm]{Theorem}

\newtheorem*{thm-nono}{Theorem}

\newtheorem{proposition}[thm]{Proposition}
\newtheorem{corollary}[thm]{Corollary}

\newtheorem{lemma}[thm]{Lemma}

\theoremstyle{definition}

\newtheorem{definition}[thm]{Definition}
\newtheorem*{theoremA}{Theorem A}
\newtheorem*{corollaryB}{Corollary B}
\newtheorem*{theoremC}{Theorem C}

\newtheorem{remark}[thm]{Remark}
\newtheorem{construction}[thm]{Construction}

\newtheorem{example}[thm]{Example}
\newtheorem*{exa-nono}{Example}
\newtheorem*{rem-nono}{Remark}

\usepackage{tikz}
\usepackage{tikz-cd}
\usetikzlibrary{calc}
\usetikzlibrary{decorations.markings}
\usetikzlibrary{decorations.pathreplacing}
\usetikzlibrary{arrows,shapes,positioning}
\tikzstyle directed=[postaction={decorate,decoration={markings,
    mark=at position #1 with {\arrow{>}}}}]
\tikzstyle rdirected=[postaction={decorate,decoration={markings,
    mark=at position #1 with {\arrow{<}}}}]
\tikzset{anchorbase/.style={baseline={([yshift=-0.5ex]current bounding box.center)}},
  tinynodes/.style={font=\tiny,text height=0.75ex,text depth=0.15ex},
  smallnodes/.style={font=\scriptsize,text height=0.75ex,text depth=0.15ex},
  >={Latex[length=1mm, width=1.5mm]}
}

\tikzset{
    partial ellipse/.style args={#1:#2:#3}{
        insert path={+ (#1:#3) arc (#1:#2:#3)}
    }
}
\newcommand{\ru}{to [out=0,in=270]}
\newcommand{\rd}{to [out=0,in=90]}
\newcommand{\ur}{to [out=90,in=180]}

\newcommand{\ul}{to [out=90,in=0]}

\newcommand{\dr}{to [out=270,in=180]}
\newcommand{\dl}{to [out=270,in=0]}

\def\R{\mathbb{R}}
\def\Z{\mathbb{Z}}

\def\C{\mathbb{C}} 

\newcommand{\kb}{{\mathbb{k}}}

\def\bd{\partial}

\newcommand{\sll}{\mathfrak{sl}}
\newcommand{\gl}{\mathfrak{gl}}
\newcommand{\glN}{\mathfrak{gl}_N}
\newcommand{\glone}{\mathfrak{gl}_1}

\newcommand{\GLN}{{\mathrm{GL}(N)}}

\newcommand{\I}{{[0,1]}} 

\newcommand{\CC}[1]{\left\llbracket #1 \right\rrbracket}

\newcommand{\KhR}{\mathrm{KhR}}

\newcommand{\fr}{\mathrm{fr}} 
\newcommand{\skeingen}{\mathcal{S}_{0,\fr}^\bullet} 
\newcommand{\skein}{\mathcal{S}_{0}^N} 
\newcommand{\fskein}{\mathcal{S}_{0,\fr}^N} 
\newcommand{\skeins}{\mathcal{S}_{0,\fr}^\Sigma}
\newcommand{\skeinsmon}{\mathcal{S}_{0,\fr}^{\lambda_i\in \Sigma}}
\newcommand{\skeine}{\mathcal{S}_{0,\fr}^\GLN}
\newcommand{\skeini}{\mathcal{S}_{0,\fr}^{N_i}}
\newcommand{\skeinitw}{\mathcal{S}_{0,\fr}^{N_i,\mathrm{tw}}}
\newcommand{\skeinone}{\mathcal{S}_{0,\fr}^{1}}
\newcommand{\surf}{S}
\newcommand{\defset}{\Sigma}

\newcommand{\frameballp}{f_+}

\newcommand{\frameballpn}{f_{\pm}}
\newcommand{\immballp}{h_+}
\newcommand{\immballn}{h_-}
\newcommand{\immballpn}{h_{\pm}}
\newcommand{\immballnp}{h_{\mp}}

\definecolor{kwcolor}{rgb}{.05, .5, .3}
\definecolor{dred}{rgb}{.7, 0, 0}
\definecolor{pwcolor}{rgb}{.15, .5, .5}

\newcommand{\revise}[1]{\textcolor[rgb]{0.00,0.00,0.00}{#1}}
\newcommand{\revadd}[1]{\textcolor[rgb]{0.00,0.00,0.00}{#1}}

\title{Invariants of surfaces in smooth $4$-manifolds from link homology} 
\author{Kim Morrison, Kevin Walker and Paul Wedrich}

\tikzset{anchorbase/.style={baseline={([yshift=-0.5ex]current bounding box.center)}}}
\usetikzlibrary{calc}
\usetikzlibrary{decorations.markings}
\usetikzlibrary{decorations.pathreplacing}
\usetikzlibrary{arrows,shapes,positioning}
\tikzstyle directed=[postaction={decorate,decoration={markings,
    mark=at position #1 with {\arrow{>}}}}]
\tikzstyle rdirected=[postaction={decorate,decoration={markings,
    mark=at position #1 with {\arrow{<}}}}]

\newcommand{\fsphere}[3]{
\draw[white, line width=1mm] (#1-#3,#2) to [out=270,in=180] (#1,#2-#3*2/3) to [out=0,in=270] (#1+#3,#2);
\draw (#1-#3,#2) to [out=270,in=180] (#1,#2-#3*2/3) to [out=0,in=270] (#1+#3,#2);
\draw[thick] (#1,#2) circle (#3);
}

\newcommand{\bsphere}[3]{
\draw[dashed] (#1-#3,#2) to [out=90,in=180] (#1,#2+#3*2/3) to [out=0,in=90] (#1+#3,#2);
}

\newcommand{\fgsphere}[3]{
\draw[white, line width=1mm] (#1-#3,#2) to [out=270,in=180] (#1,#2-#3*2/3) to [out=0,in=270] (#1+#3,#2);
\draw[opacity=.3] (#1-#3,#2) to [out=270,in=180] (#1,#2-#3*2/3) to [out=0,in=270] (#1+#3,#2);
\draw[thick, opacity=.3] (#1,#2) circle (#3);
}

\newcommand{\bgsphere}[3]{
\draw[dashed,opacity=.3] (#1-#3,#2) to [out=90,in=180] (#1,#2+#3*2/3) to [out=0,in=90] (#1+#3,#2);
}

\newcommand{\sphere}[3]{
\bsphere{#1}{#2}{#3}
\fsphere{#1}{#2}{#3}
}
\newcommand{\trefoil}[4]{
\begin{scope}[shift={(#1,#2-5.9)},scale=#3]
\draw[#4] (-0.3,5.15) to [out=150,in=270] (-1,6) to   [out=90,in=180] (0,6.8) to [out=0,in=110] (0.85,6.3)
(-0.7,6) to [out=20,in=160] (1,6) to   [out=340,in=60] (1.4,5) to [out=240,in=330] (0.3,4.85)
(0.85,5.7) to [out=240,in=30] (0,5) to   [out=210,in=300] (-1.4,5) to [out=120,in=200] (-1.15,5.85);
\end{scope}
}

\newcommand{\hopflink}[4]{
\begin{scope}[shift={(#1,#2)},scale=#3]
\begin{scope}[yscale=.66]
\draw[#4]
(.5,0) to [out=90,in=0] (-.5,1) to [out=180,in=90] (-1.5,0);
\draw[white, line width=1mm]
(-.5,0) to [out=90,in=180] (.5,1) to [out=0,in=90] (1.5,0);
\draw[#4]
(-.5,0) to [out=90,in=180] (.5,1) to [out=0,in=90] (1.5,0) to [out=270,in=0] (.5,-1) to [out=180,in=270] (-.5,0);
\draw[white, line width=1mm]
(.5,0) to [out=270,in=0] (-.5,-1);
\draw[#4]
(.5,0) to [out=270,in=0] (-.5,-1) to [out=180,in=270] (-1.5,0);
\end{scope}
\end{scope}
}
\newcommand{\hopflinkcol}[5]{
\begin{scope}[shift={(#1,#2)},scale=#3]
\begin{scope}[yscale=.66]
\draw[#5]
(.5,0) to [out=90,in=0] (-.5,1) to [out=180,in=90] (-1.5,0);
\draw[white, line width=1mm]
(-.5,0) to [out=90,in=180] (.5,1) to [out=0,in=90] (1.5,0);
\draw[#4]
(-.5,0) to [out=90,in=180] (.5,1) to [out=0,in=90] (1.5,0) to [out=270,in=0] (.5,-1) to [out=180,in=270] (-.5,0);
\draw[white, line width=1mm]
(.5,0) to [out=270,in=0] (-.5,-1);
\draw[#5]
(.5,0) to [out=270,in=0] (-.5,-1) to [out=180,in=270] (-1.5,0);
\end{scope}
\end{scope}
}

\newcommand{\unknot}[4]{
\begin{scope}[shift={(#1-1,#2)},scale=#3]
\begin{scope}[yscale=.66]
\draw[#4]
(-.5,0) to [out=90,in=180] (.5,1) to [out=0,in=90] (1.5,0) to [out=270,in=0] (.5,-1) to [out=180,in=270] (-.5,0);
\end{scope}
\end{scope}
}

\newcommand{\torushole}[4]{
\draw[#4] (#1-#3/2,#2+#3/6) to [out=300,in=180] (#1,#2-#3/6) to [out=0,in=240] (#1+#3/2,#2+#3/6);
\draw[#4] (#1-#3/3,#2) to [out=45,in=180] (#1,#2+#3/6) to [out=0,in=135] (#1+#3/3,#2);
}

\newcommand{\torus}[4]{
\draw[#4] (#1,#2) ellipse (#3 and #3*2/3);
\torushole{#1}{#2}{#3}{#4}
}

\newcommand{\lasagnafillingfigure}[1]{
    \begin{tikzpicture}[anchorbase,scale=#1]
        \begin{scope}[shift={(-18,-1)},xscale=3,yscale=12]
            \draw[dashed] (0,1) to [out=0,in=90] (1,0) to [out=270,in=0] (0,-1)
            (0,1) to [out=180,in=90] (-1,0) to [out=270,in=180] (0,-1);
        \end{scope}
        \begin{scope}[shift={(-18,-1)},xscale=9,yscale=12]
            \draw[thick] (0,-1) to [out=15,in=180] (2,-1.3) to [out=0,in=270] (3.8,0)
            to [out=90,in=0] (2,1.3) to [out=180,in=340] (0,1);
        \end{scope}
        \bsphere{-4}{-4}{10}
        \fsphere{-4}{-4}{10}
        \sphere{10.5}{-.5}{3}
        \draw[white, line width=1mm]
        (-8,4.5) to [out=80,in=290] (-7.3,11)
        (-2.9,11) to [out=280, in=180] (-1,9) to [out=0,in=260] (.7,11)
        (5.2,10.7) to [out=250, in=180] (6,8) to [out=0,in=260] (7.3,10)
        (10.3,10) to [out=250, in=90] (9,6) to [out=270,in=110] (12,0)
        (9,0) to [out=115, in=0] (0,5) to [out=180,in=20] (-3.3,4.5)
        ;
        \draw[orange, thick]
        (-8,4.5) to [out=80,in=290] (-7.3,11)
        (-2.9,11) to [out=280, in=180] (-1,9) to [out=0,in=260] (.7,11)
        (5.2,10.7) to [out=250, in=180] (6,8) to [out=0,in=260] (7.3,10)
        (10.3,10) to [out=250, in=90] (9,6) to [out=270,in=110] (12,0)
        (9,0) to [out=115, in=0] (0,5) to [out=180,in=20] (-3.3,4.5)
        ;
        \torushole{2}{7}{2}{thick, orange}
        \hopflink{10.5}{0}{1}{red,very thick}
        \begin{scope}[shift={(-6.9,4.5)},xscale=2.3,yscale=.4]
            \draw[red, very thick]
            (-.5,0) to [out=90,in=180] (.5,1) to [out=0,in=90] (1.5,0) to [out=270,in=0] (.5,-1) to [out=180,in=270] (-.5,0);
        \end{scope}
        \trefoil{3}{9}{1.5}{red,very thick}
        \hopflink{-5}{11}{1.5}{red,very thick}
        \unknot{9}{10}{1.5}{red,very thick}
        \node at (8,-8) {\small$F_1$};
        \node at (-8,3.4) {\tiny$v_1$};
        \node at (10.5,-1.5) {\tiny$v_k$};
        \end{tikzpicture}
\quad \sim \quad
\begin{tikzpicture}[anchorbase,scale=#1]
    \begin{scope}[shift={(-18,-1)},xscale=3,yscale=12]
        \draw[dashed] (0,1) to [out=0,in=90] (1,0) to [out=270,in=0] (0,-1)
        (0,1) to [out=180,in=90] (-1,0) to [out=270,in=180] (0,-1);
    \end{scope}
    \begin{scope}[shift={(-18,-1)},xscale=9,yscale=12]
        \draw[thick] (0,-1) to [out=15,in=180] (2,-1.3) to [out=0,in=270] (3.8,0)
        to [out=90,in=0] (2,1.3) to [out=180,in=340] (0,1);
    \end{scope}
    \bgsphere{-4}{-4}{10}
    \fgsphere{-4}{-4}{10}
    \sphere{1}{-5}{3}
    \sphere{10.5}{-.5}{3}
    \sphere{-8}{.6}{3}
    \draw[white, line width=1mm]
    (-9.5,.7) to [out=80,in=290] (-7.3,11)
    (-2.9,11) to [out=280, in=180] (-1,9) to [out=0,in=260] (.7,11)
    (5.2,10.7) to [out=250, in=180] (6,8) to [out=0,in=260] (7.3,10)
    (10.3,10) to [out=250, in=90] (9,6) to [out=270,in=110] (12,0)
    (9,0) to [out=115, in=0] (0,5) to [out=180,in=70] (-6.5,.7)
    (3.1,-4.7) to [out=75, in=0] (-.5,3) to [out=180,in=110] (-.8,-5.2)
    (1.5,-4.4) to [out=75, in=0] (-.1,.5) to [out=180,in=110] (.7,-5.2)
    ;
    \draw[orange, thick]
    (-9.5,.7) to [out=80,in=290] (-7.3,11)
    (-2.9,11) to [out=280, in=180] (-1,9) to [out=0,in=260] (.7,11)
    (5.2,10.7) to [out=250, in=180] (6,8) to [out=0,in=260] (7.3,10)
    (10.3,10) to [out=250, in=90] (9,6) to [out=270,in=110] (12,0)
    (9,0) to [out=115, in=0] (0,5) to [out=180,in=70] (-6.5,.7)
    (3.1,-4.7) to [out=75, in=0] (-.5,3) to [out=180,in=110] (-.8,-5.2)
    (1.5,-4.4) to [out=75, in=0] (-.1,.5) to [out=180,in=110] (.7,-5.2)
    ;
    \torushole{2}{7}{2}{thick, orange}
    \torushole{-.3}{1.5}{2}{thick, orange}
    \torus{-5}{-4}{2.3}{thick, orange}
    \trefoil{-8}{1.5}{1}{red,very thick}
    \hopflink{10.5}{0}{1}{red,very thick}
    \unknot{0.5}{-5}{.75}{red,very thick}
    \unknot{3}{-4.5}{.75}{red,very thick}
    \begin{scope}[shift={(-6.9,4.5)},xscale=2.3,yscale=.4]
        \draw[red, very thick, opacity=.5]
        (-.5,0) to [out=90,in=180] (.5,1) to [out=0,in=90] (1.5,0) to [out=270,in=0] (.5,-1) to [out=180,in=270] (-.5,0);
    \end{scope}
    \trefoil{3}{9}{1.5}{red,very thick}
    \hopflink{-5}{11}{1.5}{red,very thick}
    \unknot{9}{10}{1.5}{red,very thick}
    \node at (8,-8) {\small$F_2$};
    \node at (-8,-8) {\small$F_3$};
    \node at (-8,-.5) {\tiny$v_i$};
    \node at (1.6,-6) {\tiny$v_j$};
    \node at (10.5,-1.5) {\tiny$v_k$};
    \end{tikzpicture}
    }

\begin{document}

\maketitle

\begin{abstract}
  We construct analogs of Khovanov--Jacobsson classes and the Rasmussen
  \revadd{genus bound} for links in the boundary of any smooth oriented 4-manifold. The
  main tools are skein lasagna modules based on equivariant and deformed versions
  of $\glN$ link homology, for which we prove non-vanishing and decomposition
  results. \revadd{Along the way, we characterize precise technical conditions that allow a link homology theory to 
  extend to skein lasagna 4-manifold invariants, we establish a decomposition theorem for deformed $\glN$ skein lasagna modules, and we illustrate how Hopf link homology classes can be used to extend the functoriality of link homology theories to immersed link cobordisms.}
\end{abstract}

\section{Introduction}

In this article, we construct invariants of smooth surfaces $S$ in a given
smooth, compact, oriented $4$-manifold $W$, bounding a given framed oriented
link $L\subset \partial W$, representing a specified relative second homology
class. \revadd{We then demonstrate how these invariants yield genus bounds for such surfaces, which generalize the Rasmussen invariant for links in $S^3$.}

The main tools are \emph{skein lasagna modules}, which are invariants of
$4$-manifolds depending on a suitably functorial homology theory for links in
the three-sphere. In \cite{MR4562565} the authors introduced and constructed
such invariants based on the bigraded $\glN$ link homologies first studied by
Khovanov and Rozansky. More specifically, for a smooth, compact, oriented
$4$-manifold and, in the non-closed case, a framed, oriented link $L\subset
\partial W$ (which is allowed to be the empty link) the skein lasagna module
$\skein(W;L)$ is a $H_2(W,L)\times \Z\times \Z$-graded abelian group, which
recovers the $\glN$ link homology when $W$ is the $4$-ball.

The skein lasagna modules of $W$ are designed to serve as a natural receptacle
for an invariant of smooth surfaces in $W$ that locally behaves like the $\glN$
link homology. Here we make this idea precise and explain how a smoothly
embedded oriented surface $\surf$ determines an element of tridegree $([\surf],
\deg_q(\surf), \deg_t(\surf))\in H_2(W,L)\times \Z\times \Z$ where
\[
  \deg_q(\surf):= (1-N)\chi(\surf) - N [\surf]\cdot[\surf] 
  \quad ,\quad 
  \deg_{t}(\surf) := [\surf]\cdot[\surf]
  \, 
\] 
depend on the Euler characteristic $\chi(\surf)$ and the self-intersection of
$\surf$ computed with respect to the push-off determined by the framing of $L$
along the boundary. When $W=B^4$ and $N=2$, the surface invariant is given by
the \emph{relative Khovanov--Jacobsson classes} \cite{MR4562563} or, in other
words, by cobordism maps in Khovanov homology, which are known to detect exotica
\cite{hayden2021khovanov,hayden2023seifert}.

The key result that makes these surface invariants interesting beyond the case
of the $4$-ball is the following non-vanishing theorem that we formulate for the
$\GLN$-equivariant version of the lasagna skein module that we also construct in
this paper. To state the result, we call an oriented surface $\surf \subset W$
\emph{homologically diverse} if no union of a \revadd{nonempty} 
subset of closed
components is nullhomologous in $W$. Examples of homologically diverse surfaces
include surfaces without closed components and surfaces with a single,
homologically essential closed component.

\begin{theoremA}\label{thm:nonvanishing} Let $W$ be a smooth, compact, oriented
  $4$-manifold with a framed oriented link $L$ decorating the boundary. Then any smoothly
  embedded oriented surface $\surf\subset W$ with $\partial \surf=L$ that is
  homologically diverse defines a class of $H_2(W\revise{,} L)$-degree $[\surf]$,
  $t$-degree $\deg_t(\surf)$, and $q$-degree $\deg_q(\surf)$ of the
  $\GLN$-equivariant skein lasagna module $\skeine(W;L)$ that is
  \emph{non-trivial modulo torsion} for the action of $H^*(B\GLN)$.
  \end{theoremA}
  We will prove Theorem A in Section~\ref{sec:nonvanishing} based on a study of
  equivariant and deformed skein lasagna modules\revadd{. Before describing these in more detail, we first discuss how Theorem A} yields an Euler characteristic/genus bound for smooth surfaces in $W$.   To
  state it, we denote by $H_2(W)^L:=\delta^{-1}([L])\subset H_2(W\revise{,} L)$ the
  $H_2(W)$-torsor of the relative second homology classes in the preimage of the
  fundamental class of the link under the connecting homomorphism in the long
  exact sequence for relative homology.

\begin{corollaryB}
  Let $W$ be a smooth, compact, oriented $4$-manifold with a framed oriented
  link $L$ decorating the boundary. For $\alpha \in H_2(W)^L$ we let 
  \[q^N_{\min}(\alpha):=q^N_{\min}(W,L,\alpha)\in
\Z\cup \{-\infty\}\] denote the minimal $q$-degree of a non-zero class of $\skeine(W;L)/\textup{torsion}$ in bidegree $(\alpha,\alpha\cdot \alpha) \in
H_2(W)^L\times \Z_t$ (which is non-trivial by Theorem A).
For any smoothly embedded oriented surface $\surf\subset W$ with $\partial
\surf=L$ that is homologically diverse we then have:
\begin{align*}
 \chi(\surf) & \leq \frac{-N [\surf] \cdot [\surf]-q^N_{\min}([\surf]) }{N-1}
\end{align*}
\end{corollaryB}

\begin{exa-nono} For $W=B^4$ there is just the trivial relative second homology
class $0$ and when $N=2$ and $L$ is a knot (or link), the invariant $q^2_{\min}(0)$ is a
close relative of the \emph{Rasmussen invariant}~\cite{MR2729272} (and \cite{MR2462446}), see the
discussion about the link homology of type $\mathcal{F}_3$ in \cite{MR2232858}. For a discussion of genus bounds
derived from equivariant $\glN$ link homology we refer to \cite{MR3939052},
especially Section 5 therein. 
\end{exa-nono}

\revise{The genus bounds from Corollary B are just one way of extracting topological information about embedded surfaces from skein lasagna modules. For links in the boundary of $B^4$, the genus bound is a generalization of the Rasmussen invariant, and is sharp for positive links. It would be interesting to compare
  the bounds from Corollary B to the extended Rasmussen invariants from
  \cite{MR4541332,manolescu2023rasmussen}, which yield bounds for surfaces in
  (boundary) connected sums of $S^2\times D^2$, $S^1 \times B^3$, and \revise{$\C P^2\setminus B^4$,
  and in $\mathbb{R}P^3\times I$}, respectively. It would not be surprising if
  these bounds differ, given that they are based on a chain-level construction,
  while here we work on the level of homology. In that sense, the extended
  Rasmussen invariants from \cite{MR4541332,manolescu2023rasmussen} may be
  closer to a to-be-constructed chain-level analog of lasagna skein modules, see 
  \cite{liu2024braided} for developments in this direction.}

\begin{rem-nono}
  \revadd{ Earlier versions of this paper contained preliminary computations and speculations about applications of the techniques developed here. 
Since then, more compelling applications of these techniques have appeared in the work of other authors.
Ren and Willis \cite{ren2025khovanovhomologyexotic4manifolds} study a particular case of Lee-deformed skein lasagna modules for $\mathfrak{gl}(2)$ and use filtration arguments to define genus bounds akin to the Rasmussen invariant. This, together with various vanishing and non-vanishing results lead to the detection of exotic pairs of 4-manifolds. Sullivan \cite{sullivan2025barnatanskeinlasagnamodules} uses a different deformation of $\mathfrak{gl}(2)$ skein lasagna modules to detect exotic pairs of surfaces and to show in an example that the surfaces stay exotic after one internal stabilization.}
\end{rem-nono}

We finish this introduction with a brief overview of the proof of Theorem A. The
paper starts with the construction of a $\GLN$-equivariant skein lasagna modules
$\skeine(W;L)$ and deformed skein lasagna modules $\skeins(W;L)$, which recover
equivariant and deformed link homology theories, such as Lee homology
\cite{MR2173845}, when $W=B^4$; see also \cite{MR2232858,0402266}. The
equivariant version is defined over the ring $H^*(B\GLN)\cong \Z[e_1,\dots,e_N]$
while the deformed versions are defined over $\C$ and depend on an $N$-element
multiset $\Sigma$ of complex deformation parameters. \revadd{Along the way, we characterize the precise technical conditions that allow a link homology theory to extend to skein lasagna modules along the lines of \cite{MR4562565}, see Theorem~\ref{thm:skeinfromLH}.}
\smallskip

The proof of Theorem A now proceeds in \revise{three} main steps:
\begin{enumerate}
  \item The surface invariant in $\skeine(W;L)$ maps to the surface invariant in
  $\skeins(W;L)$ under a suitable specialization of the variables $e_i$ (Proposition~\ref{prop:equivtodef}).
  \item The deformed skein lasagna module $\skeins(W;L)$ decomposes into
  smaller-rank skein lasagna modules according to the multiplicities of the
  elements of $\Sigma$ (Theorem~\ref{thm:defskein}). If all multiplicities are
  one, the decomposition is into $\glone$ skein lasagna modules.
  \item The $\glone$ skein lasagna modules compute relative second homology (Theorem~\ref{thm:oneskein}) and
  the invariants of homologically diverse surfaces thus cannot vanish (Proposition~\ref{prop:techcond}). 
\end{enumerate}
The invariant of a homologically diverse surface $\skeine(W;L)$ cannot be
torsion, for otherwise one could find a deformation $\Sigma$ with multiplicities all one for which the
surface class would vanish in $\skeins(W;L)$, a contradiction.

We record the main result of the second step here, as it may be of independent interest:

\begin{theoremC}[Theorem~\ref{thm:defskein}]Let $W$ be a $4$-manifold and $L\subset
  \partial W$ a link in the boundary. For any multiset
  $\defset=\{\lambda_1^{N_1},\dots \lambda_n^{N_n}\}$ of complex deformation
  parameters we have an isomorphism of $H_2(W\revise{,} L)\times \Z_t$-graded $\C$-vector
  spaces:
\[
\skeins(W;L) \cong \bigoplus_{c\in C(L,\defset)} \bigotimes_{i=1}^n \skeinitw(W; L_{c^{-1}(\lambda_i)}, \C)   
\]
where \begin{itemize}
  \item $C(L,\defset)$ is the set of \emph{colorings} of the component of $L$ by
elements of $\defset$,
\item $L_{c^{-1}(\lambda_i)}$ is the sublink of $L$ colored
by $\lambda_i$ for a given $c\in C(L,\defset)$,
\item $\skeinitw(W;
L_{c^{-1}(\lambda_i)}, \C)$ is isomorphic to the $\gl_{N_i}$ skein module
$\skeini(W; L_{c^{-1}(\lambda_i)}, \C)$ with the $q$-grading forgotten.
\end{itemize}
In particular, when $W=B^4$, we obtain a functorial upgrade of the decomposition theorem of
\cite{MR3590355} for uncolored deformed link homologies: There exists a natural isomorphism of $\Z_t$-graded $\C$-linear link homologies

\begin{equation}
  \label{eq:ThmC}
    \KhR_{\defset,\fr}(L) \cong \bigoplus_{c\in C(L,\defset)} \bigotimes_{i=1}^n 
  \KhR^{\mathrm{tw}}_{N_i,\fr}(L_{c^{-1}(\lambda_i)}, \C)   
\end{equation}

where \begin{itemize}
  \item $\KhR_{\defset,\fr}$ denotes the $\Sigma$-deformed link homology with framing-adapted grading conventions,
\item $\KhR^{\mathrm{tw}}_{N_i,\fr}$ denotes the undeformed $\gl_{N_i}$ link
homology in which cobordism maps have been rescaled by a factor depending on the
Euler characteristic, $\Sigma$, and the coloring $c$.
\end{itemize}
\end{theoremC}

\revadd{
  \begin{rem-nono}
  A key idea in the proof of Theorem C is to use Hopf link homology classes to extend the functoriality of (deformed) $\glN$ link homology to certain immersed cobordisms with idempotent decorations, see Example~\ref{example:Hopfsplit}. In \cite{ren2025khovanovskeinlasagnamodules} an overlapping group of authors uses this idea to establish the functoriality of $\mathfrak{gl}_2$ homology under \emph{singular} foams which allow certain immersion points. Related recent work on functoriality of link homology under immersed cobordisms includes \cite{imori2025cobordismmapskhovanovhomology,carter2025extensionkhovanovhomologyimmersed}. We refer to \cite[\S 3.2]{carter2025extensionkhovanovhomologyimmersed} for a survey of related constructions in the literature.  
  \end{rem-nono}
}

\smallskip

\noindent \textbf{Conventions.} All $4$-manifold in this paper are assumed to be
compact, smooth and oriented. Links are framed and oriented. Surfaces in
$4$-manifolds are assumed to be oriented and \revise{normally framed\footnote{More precisely: normally $1$-framed.}}, with framing
agreeing on the boundary, unless we specifically consider an ``oriented
surface'', in which case no framing or framing-compatibility on the boundary is assumed.

\smallskip

\noindent \textbf{Acknowledgements.} PW would like to thank Ciprian Manolescu, Moritz Miltsch, Qiuyu Ren, and Isaac Sundberg for inspiring discussions concerning generalized Rasmussen invariants and Khovanov--Jacobsson classes. \revadd{We would also like to thank the anonymous referees for their careful reading and helpful comments.}

\smallskip

\noindent \textbf{Funding.}
The authors acknowledge support from the Aspen Center for Physics. PW acknowledges support from the Deutsche
Forschungsgemeinschaft (DFG, German Research Foundation) under Germany's
Excellence Strategy - EXC 2121 ``Quantum Universe'' - 390833306 and the Collaborative Research Center - SFB 1624 ``Higher structures, moduli spaces and integrability''.

\section{Equivariant and deformed skein modules}

\subsection{Equivariant and deformed link homology}

In this article we extend the construction of lasagna skein modules to
\emph{equivariant} and \emph{deformed} cases, building on equivariant and
deformed versions of the $\glN$ link homology constructed using the
Robert--Wagner evaluation of closed foams \cite{1702.04140,MR3877770}. To
explain this terminology, we recall that the ordinary $\glN$ homology of the
unknot is isomorphic\footnote{Up to regrading, this is even an isomorphism of
graded Frobenius algebras with respect to the natural Frobenius algebra
structure on the right-hand side coming from Poincar\'e duality, and on the
left-hand side corresponding to the $2$-dimensional TQFT obtained by restricting
the link homology to standardly embedded unlinks and cobordisms between them. }
to the cohomology ring $H^*(\C P^{N-1})\cong \Z[X]/X^N$. 

In equivariant $\glN$ link homology one instead
considers $\GLN$-equivariant cohomology $H^*_{\GLN}$ or cohomology equivariant
with respect to the maximal torus $T(N)\subset \GLN$. These theories are defined
over the rings:
\[ R_{T(N)}:=H^*_{T(N)} (\mathrm{point}) \cong \Z[r_1,\dots,r_N]\;, \quad R_\GLN:= H^*_\GLN
(\mathrm{point}) \cong \Z[e_1,\dots,e_N]\] 
Here the $r_j$ are classes of degree
$2$ and $e_i$ are classes of degree $2i$ that can be identified with the $i$-th
elementary symmetric polynomial $e_i(r_1,\dots,r_N)$. The corresponding unknot invariants are isomorphic to:
\[\KhR_{T(N)}(\bigcirc)\cong \frac{R_{T(N)}[X]}{(X-r_1)\cdots(X-r_N)} \;,\quad
\KhR_{\GLN}(\bigcirc)\cong \frac{R_\GLN[X]}{(\revise{\sum_{i=0}^N (-1)^{i} e_i
X^{N-i} })} \] For deformed versions, one considers a commutative ring
$\kb$ and a multiset $\defset=\{s_1,\dots, s_N\}$ of scalars in
$\kb$. The $\defset$-deformed homology has as unknot invariant the
(typically not graded, but only filtered) $\kb$-algebra:
\begin{equation}
  \label{eqn:deformedunknot}
  \KhR_{\defset}(\bigcirc)\cong 
  \frac{\kb[X]}{(\revise{\sum_{i=0}^N (-1)^{i} e_i(\defset) X^{N-i} })} 
\end{equation}
\revise{Here $X$ is of degree $2$.}
Observe that $\KhR_{\defset}(\bigcirc)$ can be obtained from
$\KhR_{\GLN}(\bigcirc)$ by tensoring with the $\kb$-$R_\GLN$-bimodule
\begin{equation}\label{eqn:RSigma}
  R_\defset:=(\kb\otimes_\Z R_\GLN)/(e_i(\defset)\otimes 1 -1 \otimes
e_i\mid 1\leq i \leq N)\end{equation}
and the ordinary unknot homology  $\KhR_{N}(\bigcirc)$ can be obtained from
$\KhR_{\GLN}(\bigcirc)$ by tensoring with the $\Z$-$R_\GLN$-bimodule
\[R_N:=\Z \otimes_\Z R_\GLN/(1 \otimes
e_i\mid 1\leq i \leq N)\cong \Z\]

More generally, the link homology constructions in \cite{1702.04140,MR3877770} associate to a
blackboard-framed link diagram $L$ a chain complex $C\KhR_\GLN(L)$ of graded
free $R_\GLN$ modules. Then one defines
\begin{align}
  \KhR_\GLN(L)&:= H_*(C\KhR_\GLN(L)), && \text{a }\Z^2\text{-graded } R_{\GLN}\text{-module}\\
  \KhR_N(L)&:= H_*(R_N\otimes_{R_{\GLN}} C\KhR_\GLN(L) ), && \text{a }\Z^2\text{-graded abelian group}\\
  \KhR_\defset(L)&:= H_*(R_\defset \otimes_{R_{\GLN}} C\KhR_\GLN(L) ), && \text{a }\Z\text{-graded, } \Z\text{-filtered } \kb\text{-module}
\end{align}
and these assignments extend to link homology theories that are functorial under
framed link cobordisms in a sense explained in Theorems~\ref{thm:skeinfromLH}
and \ref{thm:examplesofLH} below. We will refer to them as the
$\GLN$-equivariant link homology, the (ordinary) $\glN$ link homology, and the
$\defset$-deformed link homology respectively. One can also obtain a
$T(N)$-equivariant link homology by extending scalars from $R_{\GLN}$ to the
free $R_{\GLN}$-module $R_{T(N)}$, but we will not need this version in the
present paper.

\subsection{Skein modules of 4-manifolds}
In \cite{MR4562565} we defined skein modules for $4$-manifolds based on the $\glN$
link homology theories. Here we abstract from this particular setting and
present a more general framework, which will accommodate equivariant and
deformed versions of $\glN$ link homology as input. 

\begin{thm}\label{thm:skeinfromLH} Let $R$ be a commutative ring, optionally graded resp. filtered by
an abelian group $\Gamma$. Given a functorial \emph{link homology theory} for
links in $\R^3$ with values in $R$-modules, i.e. a functor:
\[
\begin{Bmatrix}
\textrm{link embeddings in oriented } B\cong \R^3
\\
\textrm{link cobordisms in oriented }
Y \cong \R^3\times \I \textrm{ up to isotopy rel } \bd
\end{Bmatrix}
\xrightarrow{H}
\begin{Bmatrix}
\Z\textrm{-graded, } \Gamma\textrm{-graded/filtered }  R\textrm{-modules}
\\
\Z\textrm{-pres., } \Gamma\textrm{-hom./filt.} R\textrm{-morphisms}
\end{Bmatrix}
\]
which additionally
\begin{enumerate}
  \item is (lax) monoidal under disjoint union: $H(L_1)\otimes_R H(L_2)\to H(L_1 \sqcup L_2)$, and
  \item satisfies the sweep-around move \cite[(1.1)]{MR4562565}, and
  \item the trace of the $2\pi$ rotation of $\R^3$, which generates
  $\pi_1(SO(3))$, acts by the identity on $H$, 
\end{enumerate}
then $H$ extends to:
\begin{itemize}
  \item a link homology for links in $3$-spheres \cite[Definition 4.8]{MR4562565} with
  values in $\Z$-graded, $\Gamma$-graded/filtered $R$-modules,
  \item an algebra for the lasagna operad \cite[Section 5.1]{MR4562565} valued in
  $\Z$-graded, $\Gamma$-graded/filtered $R$-modules,
  \item a $(4+\epsilon)$-dimensional TQFT of skein modules associated to pairs
  $(W,L)$ of oriented $4$-manifolds $W$ with links $L\subset \partial W$, valued
  in  $\Z$-graded, $\Gamma$-graded/filtered $R$-modules as in \cite[Section 5.2]{MR4562565}.
\end{itemize}
\revise{Here the morphisms of the target category of $H$ are given by $R$-linear maps that preserve the $\Z$-grading and are homogeneous resp. filtered with respect to the $\Gamma$-grading.}
\end{thm}
\begin{proof} The proof follows the main narrative of \cite{MR4562565}, abstracted
from $\glN$ homology and with coefficients generalized from $\Z$ to $R$. First,
the requirements 2 and 3 are used to extend $H$ to a functorial invariant of
links in $S^3$. Second, requirements 1 and 2 are used to extend $H$ to an
algebra for the lasagna operad, which associates maps to surfaces embedded in
the complement of $4$-balls inside a larger $4$-ball. Finally, the lasagna
algebra is used to define the skein relations in the $R$-linear skein modules
associated to $(W,L)$. 
\end{proof}

\begin{theorem}\label{thm:examplesofLH} The requirements of Theorem~\ref{thm:skeinfromLH} are satisfied for:
  \begin{itemize}
    \item The $\GLN$-equivariant link homology $\KhR_{\GLN}$ with
    $\Gamma=\Z$ and graded ring $R_\GLN$.
    \item The $\glN$ link homology $\KhR_N$ with $\Gamma=\Z$ and graded ring
    (concentrated in degree zero) $R_N\cong\Z$.
    \item The $\defset$-deformed link homology $\KhR_\defset$ with $\Gamma=\Z$ and
    filtered ring (concentrated in degree zero) $R_\defset$.
  \end{itemize}
\end{theorem}
\begin{proof}
  The proof for the ordinary $\glN$ link homology was given in
  \cite{MR4562565} and the other cases follow mutatis mutandis. The
  functoriality of the $\GLN$-equivariant and deformed link homologies in
  $\R^3$, constructed via the Robert--Wagner foam evaluation \cite{1702.04140}
  was proven in \cite{MR3877770}. Two important properties are manifestly
  visible in this construction. First, the link homologies are computed as
  homologies of chain complexes of graded free modules over the respective
  ground ring $R$ (in the case of the deformed homology: with filtered
  differentials) \cite[Remark 3.30]{1702.04140}, which implies requirement 1 in
  Theorem~\ref{thm:skeinfromLH}. 

  Second, the complexes associated to link diagrams, which compute the link
  homologies under consideration, have chain modules that are spanned by foams
  from the empty web to various closed webs obtained from resolving the link
  diagram, and the differentials are given by postcomposition by other foams.
  Since these foams are considered up to isotopy relative to the boundary,
  requirement 3 in Theorem~\ref{thm:skeinfromLH} is automatically satisfied.
  
  The final and only substantial requirement is the sweep-around move, which was
  proved for the ordinary $\glN$ link homology in \cite[Section
  3]{MR4562565}. The proof immediately generalizes to the other cases;
  the only visible, although immaterial difference is in a polynomial expression
  describing the action of the Reidemeister I chain maps in \cite[Lemma
  3.9]{MR4562565}. 
\end{proof}

\begin{remark}[Remark on grading conventions]
  \label{rem:gradconv}
In \cite{MR4562565} we have normalized the $\glN$ link homology so
  that a framing change of $\pm 1$ on a link corresponds to a shift $q^{\pm 1
  \mp N}$. This can be computed from the local crossing relation:
  \begin{equation}
    \label{eqn:crossingsold}
    \CC{\begin{tikzpicture}[anchorbase,scale=.2]
      \draw [very thick, ->] (1,-2)to(1,-1.7) to [out=90,in=270] (-1,1.7) to (-1,2);
      \draw [white, line width=.15cm] (-1,-2) to (-1,-1.7) to [out=90,in=270] (1,1.7) to (1,2);
      \draw [very thick, ->] (-1,-2) to (-1,-1.7) to [out=90,in=270] (1,1.7) to (1,2);
    \end{tikzpicture}
    }
     \quad = \quad
    q\;
    \begin{tikzpicture}[anchorbase,scale=.2]
      \draw [very thick, ->] (1,-2)to(1,-1.7) to [out=90,in=315] (0,-.5) (0,.5) to [out=45,in=270] (1,1.7) to (1,2);
      \draw[double] (0,.5) to (0,-.5);
      \draw [very thick, ->] (-1,-2) to(-1,-1.7) to [out=90,in=225] (0,-.5) (0,.5) to [out=135,in=270] (-1,1.7) to (-1,2);
    \end{tikzpicture}
    \to
    \uwave{
    \begin{tikzpicture}[anchorbase,scale=.2]
      \draw [very thick, ->] (1,-2) to (1,2);
      \draw [very thick, ->] (-1,-2) to (-1,2);
    \end{tikzpicture}
    }
    , \qquad
    \CC{\begin{tikzpicture}[anchorbase,scale=.2]
      \draw [very thick, ->] (-1,-2) to (-1,-1.7) to [out=90,in=270] (1,1.7) to (1,2);
      \draw [white, line width=.15cm] (1,-2)to(1,-1.7) to [out=90,in=270] (-1,1.7) to (-1,2);
      \draw [very thick, ->] (1,-2)to(1,-1.7) to [out=90,in=270] (-1,1.7) to (-1,2);
    \end{tikzpicture}
    }
     \quad = \quad
    \uwave{
    \begin{tikzpicture}[anchorbase,scale=.2]
      \draw [very thick, ->] (1,-2) to (1,2);
      \draw [very thick, ->] (-1,-2) to (-1,2);
    \end{tikzpicture}
    }
    \to
    q^{-1}\;
    \begin{tikzpicture}[anchorbase,scale=.2]
      \draw [very thick, ->] (1,-2)to(1,-1.7) to [out=90,in=315] (0,-.5) (0,.5) to [out=45,in=270] (1,1.7) to (1,2);
      \draw[double] (0,.5) to (0,-.5);
      \draw [very thick, ->] (-1,-2) to(-1,-1.7) to [out=90,in=225] (0,-.5) (0,.5) to [out=135,in=270] (-1,1.7) to (-1,2);
    \end{tikzpicture}
     \end{equation}
  The corresponding skein modules of $4$-manifolds were denoted $\skein(W;L)$ in
  \cite{MR4562565}. 

  In the present paper it is more convenient (see Remark~\ref{rem:whyconv}) to
  work with a global regrading using the local crossing relation:
  
  \begin{equation}
    \label{eqn:crossingsnew}
    \CC{\begin{tikzpicture}[anchorbase,scale=.2]
      \draw [very thick, ->] (1,-2)to(1,-1.7) to [out=90,in=270] (-1,1.7) to (-1,2);
      \draw [white, line width=.15cm] (-1,-2) to (-1,-1.7) to [out=90,in=270] (1,1.7) to (1,2);
      \draw [very thick, ->] (-1,-2) to (-1,-1.7) to [out=90,in=270] (1,1.7) to (1,2);
    \end{tikzpicture}
    }_\fr
     \quad = \quad
     \uwave{
    \begin{tikzpicture}[anchorbase,scale=.2]
      \draw [very thick, ->] (1,-2)to(1,-1.7) to [out=90,in=315] (0,-.5) (0,.5) to [out=45,in=270] (1,1.7) to (1,2);
      \draw[double] (0,.5) to (0,-.5);
      \draw [very thick, ->] (-1,-2) to(-1,-1.7) to [out=90,in=225] (0,-.5) (0,.5) to [out=135,in=270] (-1,1.7) to (-1,2);
    \end{tikzpicture}
    }
    \to
      q^{-1}\;
    \begin{tikzpicture}[anchorbase,scale=.2]
      \draw [very thick, ->] (1,-2) to (1,2);
      \draw [very thick, ->] (-1,-2) to (-1,2);
    \end{tikzpicture}
    , \qquad
    \CC{\begin{tikzpicture}[anchorbase,scale=.2]
      \draw [very thick, ->] (-1,-2) to (-1,-1.7) to [out=90,in=270] (1,1.7) to (1,2);
      \draw [white, line width=.15cm] (1,-2)to(1,-1.7) to [out=90,in=270] (-1,1.7) to (-1,2);
      \draw [very thick, ->] (1,-2)to(1,-1.7) to [out=90,in=270] (-1,1.7) to (-1,2);
    \end{tikzpicture}
    }_\fr
     \quad = \quad
     q\;
    \begin{tikzpicture}[anchorbase,scale=.2]
      \draw [very thick, ->] (1,-2) to (1,2);
      \draw [very thick, ->] (-1,-2) to (-1,2);
    \end{tikzpicture}
    \to
    \uwave{
    \begin{tikzpicture}[anchorbase,scale=.2]
      \draw [very thick, ->] (1,-2)to(1,-1.7) to [out=90,in=315] (0,-.5) (0,.5) to [out=45,in=270] (1,1.7) to (1,2);
      \draw[double] (0,.5) to (0,-.5);
      \draw [very thick, ->] (-1,-2) to(-1,-1.7) to [out=90,in=225] (0,-.5) (0,.5) to [out=135,in=270] (-1,1.7) to (-1,2);
    \end{tikzpicture}
    }
     \end{equation}
In these conventions, a framing shift of $\pm 1$ on a link corresponds to a
shift of $ q^{\mp N} t^{\pm 1}$ (when using the convention that the differential
raises the $t$-degree). Since the two associated versions of $\glN$ homology
differ by a grading shift, the version based on \eqref{eqn:crossingsnew}
inherits the same functoriality properties under \emph{framed} link cobordisms
and hence also give rise to skein modules for $4$-manifolds as described in
Theorem~\ref{thm:skeinfromLH}. To distinguish the associated link homologies and
skein modules for $4$-manifolds from their counterparts defined using the other
convention, we will mark them with a subscript $\fr$, e.g. $\KhR_{\GLN, \fr}$,
$\fskein(W;L)$. 

The reader be warned
that the papers \cite{2020arXiv200908520M,
MR4589588} give formulas to compute
$\skein(W;L)$,
\revadd{rather than $\fskein(W;L)$,} from a handle decomposition. Analogous formulas for $\fskein(W;L)$
can be obtained after regrading appropriately.  
\end{remark}

We are now ready to set notation for the $4$-manifold skein modules associated
with the equivariant, the ordinary, and the deformed $\glN$ link homologies with
respect to the grading conventions from Remark~\ref{rem:gradconv}.

\begin{definition} Let $W$ be a $4$-manifold and $L\subset \partial W$ a link.
For $N\geq 1$, a commutative ring $\kb$ and a multiset $\defset$ of $N$
elements of $\kb$ we denote by
\begin{itemize}
  \item $\skeine(W;L)$, the skein module of $(W,L)$ constructed from $\KhR_{\GLN, \fr}$,
  \item  $\fskein(W;L)$, the skein module of $(W,L)$ constructed from $\KhR_{N, \fr}$,
  \item $\skeins(W;L)$, the skein module of $(W,L)$ constructed from $\KhR_{\defset, \fr}$.
\end{itemize}
via Theorems~\ref{thm:skeinfromLH} and \ref{thm:examplesofLH} using the grading
conventions of Remark~\ref{rem:gradconv}.
\end{definition}

For the readers convenience, we briefly recall the explicit description of these
skein modules. More details appear in \cite[Section 5.2]{MR4562565}.
For $\bullet\in \{\GLN, N, \defset\}$ the associated skein module is defined as a
quotient of a homologically $\Z$-graded and quantum $\Z$-graded/filtered free
$R_\bullet$-module:

\[
	\skeingen(W; L) := R_\bullet\{\text{lasagna fillings } F \text{ of } W \text{ with boundary } L\}/\sim
\]
spanned by \emph{lasagna fillings}, which consist of 
	\begin{itemize}
	\item a finite collection of `small' 4-balls $B_i$ embedded in the interior of $W$;
	\item a framed oriented surface $\surf$ properly embedded in $X \setminus \sqcup_i B_i$, 
	meeting $\bd W$ in $L$ and
	meeting each $\bd B_i$ in a link $L_i$; and
	\item for each $i$, a homogeneous label $v_i \in \KhR_\bullet(\bd B_i, L_i)$.
	\end{itemize}
Lasagna fillings are considered multilinear in the input labels $v_i$ and up to
the equivalence relation $\sim$ which is the linear and transitive closure of
the relation generated by declaring two lasagna fillings $F_1,F_2$ equivalent,
in symbols $F_1\sim F_2$, if $F_1$ has an input ball $B_1$ with label $v_1$, and
$F_2$ can be obtained from $F_1$ by replacing $B_1$ with a third lasagna filling
$F_3$ of a 4-ball such that the label $v_1$ is computed as the lasagna algebra
evaluation of $F_3$, followed by an isotopy rel boundary. This relation is
illustrated in Figure~\ref{fig:lasagnafillingequiv}, reproduced from
\cite{MR4562565}.

\begin{figure}[ht]
	\[
	\lasagnafillingfigure{.15}
	\]
	\caption{}
	\label{fig:lasagnafillingequiv}
\end{figure}

The skein module $\skeingen(W; L)$ is graded by $H_2(W\revise{,} L)\times \Z_t$ and
additionally $\Z_q$\revise{-}graded/filtered. A lasagna filling $F$ with input labels
$v_i$ and underlying surface $\surf$ has $\Z_q\times \Z_t$ bidegree 
\[
  \deg_{q,t}(F) := \left((1-N) \chi(\surf), 0 \right) + \sum_{i} \deg_{q,t}(v_i)\in \Z_q\times \Z_t
\]
and $H_2(W\revise{,} L)$-degree given by the homology class of the surface obtained by
capping off $\surf$ by Seifert surfaces for the input links $L_i$, which by
abuse of notation we denote by $[\surf]$. 

As mentioned in the introduction,
with regard to the grading by relative second homology, the skein module
$\skeingen(W; L)$ is concentrated in the $H_2(W)$-torsor
$H_2(W)^L:=\delta^{-1}([L])\subset H_2(W\revise{,} L)$ of the relative second homology
classes in the preimage of the fundamental class of the link under the
connecting homomorphism in the long exact sequence for relative homology.

\subsection{Skeins from surfaces}
\label{sec:skeinsfromsurfaces}

\begin{definition}[Framing-changing input balls]
  The $\glN$ homology of the unknot with framing $+1$ is $N$-dimensional,
  concentrated in bidegrees $t q^{-2N+1+2i}$ for $0\leq i \leq N-1$. We fix the
  bidegree $t q^{-2N+1}$ class corresponding to the image of the disk-class
  under the usual Reidemeister I chain map and denote it $\frameballp$. An input sphere
  in a lasagna filling, whose link is the positively framed unknot decorated
  with the class $\frameballp$ will be called a \emph{positive framing-changing
  input ball}. Analogously we define negative framing-changing input balls of
  bidegree $t^{-1} q^{2N-1}$. The relations in the skein modules $\fskein(W;L)$
  ensure that positive and negative framing-changing input balls on the same
  component of a lasagna sheet cancel.
  \end{definition}
  
  \begin{definition}[Skein classes from surfaces]
    \label{def:skeinfromsurface}
  For an oriented surface $\surf$ in $(W,L)$ we denote by $[\surf]\cdot[\surf]$  
  the self-intersection of $\surf$ with respect to the push-off along the
  boundary given by the framing on $L$. As the notation suggests, this only depends on
  the relative second homology class $[\surf]\in H_2(W)^L \subset H_2(W\revise{,} L)$ and the given
  framing of the link $L$. By removing a suitable number of disks $\sqcup_{i_j} D^2$
  from \revise{each connected component $\surf_j$ of }$\surf$ and filling them with framing-changing input balls of total
  framing change $\revise{[\surf_j]\cdot[\surf_j]}$, we can arrange the complement
  $\surf\setminus{\sqcup_{\revise{j,i_j}} D^2}$ to \revise{admit} a normal framing compatible with the
  framing of $L$ and the framed unknots on the framing-changing input balls. \revise{For any choice of positioning of these input balls and normal framing of the complement, we thus} obtain a class in $\fskein(W;L)$ that we again denote by $\surf$\revise{. Indeed, this class is well-defined and independent of the choices made since framing-changing input balls can be isotoped and cancelled when appearing with opposite sign on a lasagna sheet, and the possible compatible framings of the complement surface are related by these operations. The
  tridegree of the class $\surf$} is $([\surf],\deg_q(\surf),\deg_t(\surf))\in H_2(W,L)\times
  \Z\times \Z$ where
  \begin{align*}
    \deg_{t}(\surf) &:= [\surf]\cdot[\surf]\, ,\\
    \deg_q(\surf)&:= (1-N)\chi(\surf) - N [\surf]\cdot[\surf]\, .
  \end{align*}
   Similarly, a surface $\surf$ also determines a class in any deformed skein
  module, which will only have a bidegree $([\surf],\deg_t(\surf))$.
  \end{definition}
  
  \begin{example} The $2$-sphere $\C P^1 \subset \C P^2$ has self-intersection
    $1$, hence does not represent an element of the skein module $\fskein(\C
    P^2;\emptyset)$ on its own. After inserting a framing-changing input ball, we
    obtain a lasagna filling representing a class of tridegree $([\C P^1], 2-3N,
    1)$. On the other hand, the $2$-sphere $\C P^1 \subset \overline{\C P^2}$ has
    self-intersection $-1$, hence we obtain a class of tridegree $([\C P^1], 2-N,
    -1)$. 
  \end{example}

\section{Non-vanishing and genus bounds}

\subsection{Decomposing deformed skein modules}
In this section we fix as multiset $\defset$ of complex numbers $\lambda_1,\dots,
\lambda_n$ (assumed to be pairwise distinct) with multiplicies $N_i$ summing to
$N=\sum_i N_i$. We will use the notation $\defset=\{\lambda_1^{N_1},\dots
\lambda_n^{N_n}\}$ for this multiset.

Following \eqref{eqn:deformedunknot} the
$\defset$-deformed unknot homology over $\kb=\C$ is the $\C$-algebra

\begin{align*}  
  \KhR_{\defset}(\bigcirc)\cong 
\frac{\C[X]}{\prod_{i=1}^n (X-\lambda_i)^{N_i}}
\cong
\prod_{i=1}^n \frac{\C[X]}{(X-\lambda_i)^{N_i}}
\end{align*}
according to the Chinese Remainder Theorem. The factors can again be identified
with $\gl_{N_i}$ unknot homologies (with the $q$-grading forgotten). Let
$e(\lambda_i)\in \KhR_{\defset}(\bigcirc)$ denote the idempotent projecting onto
the $\lambda_i$-summand. In the skein module $\skeins$, we may interpret
$e(\lambda_i)$ as the label on an input ball in a lasagna filling $F$,
intersecting the surface of $F$ in an unknot. Since the $e(\lambda_i)$ are
idempotent, the location and number (as long as it is positive) of such local
decorations on a component of the surface in $F$ is immaterial, so we consider
the entire component as $\lambda_i$-colored.

\begin{definition}Let $W$ be a $4$-manifold and $L\subset \partial W$ a link in the boundary.
  \begin{itemize}
   \item We denote by $C(L,\defset)=\{\lambda_1,\dots \lambda_n\}^{\pi_0(L)}$ the
  \emph{colorings} of the component of $L$ by elements of $\defset$. 
  \item For a fixed coloring $c\in C(L,\defset)$ we denote by
  $L_{c^{-1}(\lambda_i)}$ the \emph{color-homogeneous sublink} of $L$ given by
  the union of those components of $L$ that are colored $\lambda_i$ by $c$.
  \item For a fixed coloring $c\in C(L,\defset)$ we will denote by
  $\skeins(W;c(L))$ the $H_2(W\revise{,} L)\times \Z_t$-graded $\C$-vector subspace of
  $\skeins(W;L)$ spanned by lasagna fillings whose underlying surface has
  components $\lambda_i$-colored whenever they bound a component of $L$ that is
  $\lambda_i$-colored by $c$.
  \item In the special case of $W=B^4$, we have an isomorphism
  \begin{equation}
    \label{eqn:B4id}
   \phi\colon \KhR_\defset(L) \xrightarrow{\cong}  \skeins(B^4;L)
  \end{equation}
    as in \cite[Example 5.6]{MR4562565}, which sends a class $v\in
    \KhR_\defset(L)$ to the lasagna filling of $B^4$ given by a single
    $L$-decorated concentric input ball with label $v$ and radial surface
    $L\times I$. For the coloring $c$ we then define $\KhR_\defset(c(L)):=
    \phi^{-1}(\skeins(B^4;c(L)))$, so that $\phi$ restricts to an isomorphism:
    \begin{equation}
      \label{eqn:B4idc}
      \KhR_\defset(c(L)) \cong \skeins(B^4;c(L))
    \end{equation}
   \end{itemize}
\end{definition}

\begin{lemma}
  \label{lem:onecolattime}
  Let $W$ be a $4$-manifold and $L\subset \partial W$ a link in the boundary.
  Then we have a decomposition of $H_2(W\revise{,} L)\times \Z_t$-graded $\C$-vector spaces:
  \[
    \skeins(W;L) \cong \bigoplus_{c\in C(L,\defset)}  \skeins(W;c(L))
    \]
and in the case of $W=B^4$:
\[
  \KhR_\defset(L) \cong \bigoplus_{c\in C(L,\defset)}  \KhR_\defset(c(L))
\]
\end{lemma}
\begin{proof}
This is a direct consequence of the definitions and the fact that any disk in the
lasagna sheet can be replaced by a sum of unknot input balls with labels
$e(\lambda_i)$ for $1\leq i \leq n$, i.e. any sheet splits into the sum of its
colored components. \revise{This shows that the left-hand sides are sums as indicated on the right-hand sides. The directness of the sums follows from the orthogonality of the idempotents $e(\lambda_i)$. In fact, projections onto different summands can be obtained by gluing on collars $\partial W\times I$ with idempotent-decorated identity cobordisms $L\times I$.}
\end{proof}

Our goal is to decompose the summands $\skeins(W;c(L))$ further into tensor
products of $\gl_{N_i}$ skein modules for the sublinks $L_{c^{-1}(\lambda_i)}$:

\begin{theorem}\label{thm:defskein} Let $W$ be a $4$-manifold and $L\subset
  \partial W$ a link in the boundary. For any multiset
  $\defset=\{\lambda_1^{N_1},\dots \lambda_n^{N_n}\}$ of complex deformation
  parameters we have an isomorphism of $H_2(W\revise{,} L)\times \Z_t$-graded $\C$-vector
  spaces:

\begin{equation}
  \label{eqn:chromatographyskein}
\skeins(W;L) \cong \bigoplus_{c\in C(L,\defset)} \bigotimes_{i=1}^n \skeinitw(W; L_{c^{-1}(\lambda_i)}, \C)   
\end{equation}
where $\skeinitw(W; L_{c^{-1}(\lambda_i)}, \C)$ is isomorphic to the $\gl_{N_i}$
skein module $\skeini(W; L_{c^{-1}(\lambda_i)}, \C)$ with the $q$-grading forgotten.
\end{theorem}
\begin{remark}
The direct sum in Theorem~\ref{thm:defskein} can be restricted to the index set
of those colorings $c\in C(L,\defset)$ of link components by elements of
$\defset$, such that every sublink $L_{c^{-1}(\lambda_i)}$ is nullhomologous in
$W$. Otherwise the corresponding tensor factor $\skeinitw(W;
L_{c^{-1}(\lambda_i)},\revise{\C})$ and thus the entire direct summand is zero.
\end{remark}

\begin{remark}\label{rem:whyconv}
 The isomorphism in Theorem~\ref{thm:defskein} is our main reason to prefer the
 grading convention for $\fskein$ over those of $\skein$. The analogous
 statement for $\skein$ would require additional grading shifts on the summands
 on the right-hand side, computed from linking numbers of the color-homogeneous
 sublinks.
\end{remark}

Before we prove Theorem~\ref{thm:defskein}, we consider the special case of
$W=B^4$, which appears in \cite{MR3590355}. In this setting, the relevant
isomorphism simply relates the $\defset$-deformed link homology with
$\mathfrak{gl}_{N_i}$ link homologies of sublinks. 

\begin{theorem}\label{thm:deflinks}
  Let $L\subset S^3$ be a link and $\defset=\{\lambda_1^{N_1},\dots
  \lambda_n^{N_n}\}$ a multiset of complex deformation parameters. Then there
  exists an isomorphism of $\Z_t$-graded $\C$-vector spaces: 
  \begin{equation}
    \label{eqn:linkdecomp}
\KhR_{\defset,\fr}(L) \cong \bigoplus_{c\in C(L,\defset)} \bigotimes_{i=1}^n \KhR_{N_i,\fr}(L_{c^{-1}(\lambda_i)}, \C)   
  \end{equation}
\end{theorem}
\begin{proof}
This is a direct consequence of the version of the main result (Theorem 1.1) of
\cite{MR3590355}, specialized to links colored with the vector representation
(instead of more general fundamental representations), in slightly different
grading conventions (invisible in the isomorphism), and using the construction
via Robert--Wagner foams instead of foams determined via categorified quantum
groups, see also \cite[Section 4.1]{1702.04140}.
\end{proof}

The analogous result for skein modules, Theorem~\ref{thm:defskein}, would follow
from the local result in $B^4$, provided the isomorphisms in
\eqref{eqn:linkdecomp} assemble into a natural isomorphism of link homology
theories, i.e. are if they are natural with respect to link cobordism. It turns
out that this is actually not the case, although for a boring reason: one can
already observe in small examples that cobordism maps have to be rescaled by a
nonzero factor depending on the Euler characteristic of the cobordism and
$\defset$. For example, (dotted) $e(\lambda_i)$-colored $2$-spheres evaluate to
scalars depending on $\defset$ instead of just $0$ or $1$. 

Along the way towards proving \eqref{eqn:chromatographyskein}, we will establish
a natural isomorphism of link homologies, in which the tensor factors on the
right-hand side in \eqref{eqn:linkdecomp} are replaced by versions of the
$\gl_{N_i}$ link homology, twisted by an invertible scalar depending on
$\defset$. The resulting decomposition result is strictly stronger\footnote{In
fact, it was this theorem that first lead to a proof of the functoriality of
$\glN$ homology in \cite{MR3877770} and only now that this has been established,
it makes sense to ask the naturality question.} than Theorem~\ref{thm:deflinks}
from \cite{MR3590355} 
for \eqref{eqn:linkdecomp}. Rather than refining the proof from
\cite{MR3590355} we will give an entirely new argument, making use of the
functoriality result from \cite{MR3877770}. Since the latter builds on
\cite{MR3590355}, our approach here is not independent from that line of work.

\begin{example}[Immersion point input balls]\label{example:Hopfsplit} Consider a positive Hopf link $H$ and a
coloring $c$ of one component by $\lambda_i$ (green, lighter) and the other by
$\lambda_j$ (red, darker) for $j\neq i$. 
\revise{We compute the summand of $\KhR_{\defset,\fr}\revise{(H)}$
corresponding to $c$, also denoted $\skeins(B^4;c(H))$, using the
following complex of (idempotent-colored) webs and foams}
\begin{equation}
  \label{eqn:deformedHopf}  
\xy
(-25,0)*{\begin{tikzpicture} [xscale=-.4,yscale=.4]
 \draw[dred, very thick, directed=.6] (1.7,1.3) to [out=315,in=180] (2,1) to [out=0,in=180] (3,2) to [out=0,in=270] (3.5,2.5) to [out=90,in=0] (3,3) to (1,3) to [out=180,in=90] (0.5,2.5) to [out=270,in=180] (1,2) to [out=0,in=135] (1.3,1.7);
 \draw[green, very thick,directed=.4] (2.7,1.3) to [out=315,in=180] (3,1) to [out=0,in=90] (3.5,0.5) to [out=270,in=0] (3,0) -- (1,0) to [out=180,in=270] (.5,0.5) to [out=90,in=180] (1,1) to [out=0,in=180] (2,2) to [out=0,in=135] (2.3,1.7);
\end{tikzpicture}};
(80,0)*{
\begin{tikzpicture} [scale=.4]
 \draw[very thick, dred] (3.5,2.5) to [out=90,in=0] (3,3) to (1,3) to [out=180,in=90] (0.5,2.5);
 \draw[very thick, green] (3.5,0.5) to [out=270,in=0] (3,0) -- (1,0) to [out=180,in=270] (.5,0.5);
  \draw[very thick, dred] (0.5,2.5) to [out=270,in=180] (1,2) (3,2) to [out=0,in=270] (3.5,2.5) (1,1) to (3,1);
 \draw[very thick,green]  (.5,0.5) to [out=90,in=180] (1,1) (3,1) to [out=0,in=90] (3.5,0.5) (1,2) to (3,2);
\end{tikzpicture}
};
(88,0)*{
\begin{tikzpicture} [scale=.4]
 \draw[opacity=0] (0,0) to (0,-5);
\end{tikzpicture}
};
(20,5.5)*{\begin{tikzpicture} [scale=.5]
 \draw[->] (0,0) -- (1.5,0.5);
  \end{tikzpicture}};
 (20,-5.5)*{\begin{tikzpicture} [scale=.5]
 \draw[->] (0,0) -- (1.5,-0.5);
  \end{tikzpicture}};
 (60,5.5)*{\begin{tikzpicture} [scale=.5]
 \draw[->] (0,0) -- (1.5,-0.5);
  \end{tikzpicture}};
 (60,-5.5)*{\begin{tikzpicture} [scale=.5]
 \draw[->] (0,0) -- (1.5,0.5);
 \end{tikzpicture}};
(40,10)*{
\begin{tikzpicture} [scale=.4]
  \draw[very thick, dred] (3,2) \ru (3.5,2.5) to [out=90,in=0] (3,3) to (1,3) to [out=180,in=90] (0.5,2.5) (1.5,1.5)\dr (2,1) to (3,1);
  \draw[very thick, green] (3,1) \rd (3.5,0.5) to [out=270,in=0] (3,0) -- (1,0) to [out=180,in=270] (.5,0.5)  (1.5,1.5)\ur (2,2) to (3,2);
 \draw[double] (1,1.5) to (1.5,1.5);
 \draw[very thick,dred] (.5,2.5) to [out=270,in=90] (1,1.5);
 \draw[very thick,green] (1,1.5)  to  [out=270,in=90]  (.5,.5);
\end{tikzpicture}
};
(40,-10)*{
\begin{tikzpicture} [scale=.4]
  \draw[very thick, dred] (3.5,2.5) to [out=90,in=0] (3,3) to (1,3) to [out=180,in=90] (0.5,2.5) \dr (1,2) (2.5,1.5) \dl (2,1) to (1,1);
  \draw[very thick, green] (3.5,0.5) to [out=270,in=0] (3,0) -- (1,0) to [out=180,in=270] (.5,0.5) \ur (1,1) (2.5,1.5) \ul (2,2) to (1,2);
\draw[double] (2.5,1.5) to (3,1.5);
\draw[very thick, dred] (3.5,2.5) to [out=270,in=90] (3,1.5);
\draw[very thick,green] (3,1.5)  to  [out=270,in=90] (3.5,.5);
\end{tikzpicture}
};
(0,0)*{
\begin{tikzpicture} [scale=.4]
  \draw[very thick, dred] (3.5,2.5) to [out=90,in=0] (3,3) to (1,3) to [out=180,in=90] (0.5,2.5);
  \draw[very thick, green] (3.5,0.5) to [out=270,in=0] (3,0) -- (1,0) to [out=180,in=270] (.5,0.5);
 \draw[very thick, dred] (3.5,2.5) to [out=270,in=90] (3,1.5) (.5,2.5) to [out=270,in=90] (1,1.5);
 \draw[very thick,green] (1,1.5)  to  [out=270,in=90]  (.5,.5) (3,1.5)  to  [out=270,in=90] (3.5,.5);
  \draw[very thick,dred] (1.5,1.5) \dr (2,1) \ru (2.5,1.5);
  \draw[very thick,green] (1.5,1.5) \ur (2,2) \rd (2.5,1.5);
  \draw[double] (1,1.5) to (1.5,1.5);
  \draw[double] (2.5,1.5) to (3,1.5);
\end{tikzpicture}
};
(32,10)*{q^{-1}};
(32,-10)*{q^{-1}};
(72,0)*{q^{-2}};
(40,0)*{\bigoplus};
\endxy 
\end{equation}
where the leftmost term is in $t$-degree zero. The other terms are incompatibly
colored and thus are isomorphic to zero in the $\defset$-deformed foam category.
We now fix as basis element for the $1$-dimensional $\C$-vector space
$\skeins(B^4;c(H))$ the $t$-degree zero class $\immballp$ represented by the
idempotent labelled foam pictured in the following figure. 
\[\immballp:=
\raisebox{-0.5\height}{\includegraphics{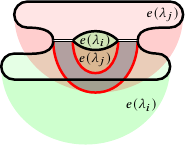}}
  \quad, \quad
  \raisebox{-0.5\height}{\includegraphics{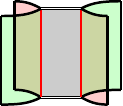}}
\] After bending, this foam (in its form shown on the right above) can also be
interpreted as the only non-zero component in the following isomorphism between
the positive and the negative idempotent-colored crossing complexes:

\[
  \CC{\begin{tikzpicture}[anchorbase,scale=.2]
    \draw [very thick, green, ->] (1,-2)to(1,-1.7) to [out=90,in=270] (-1,1.7) to (-1,2);
    \draw [white, line width=.15cm] (-1,-2) to (-1,-1.7) to [out=90,in=270] (1,1.7) to (1,2);
    \draw [very thick,dred, ->] (-1,-2) to (-1,-1.7) to [out=90,in=270] (1,1.7) to (1,2);
  \end{tikzpicture}
  }_\fr
   \quad = \quad 
   \uwave{
  \begin{tikzpicture}[anchorbase,scale=.2]
    \draw [very thick,green, ->] (1,-2)to(1,-1.7) to [out=90,in=315] (0,-.5)(0,.5) to [out=135,in=270] (-1,1.7) to (-1,2);
    \draw[double] (0,.5) to (0,-.5);
    \draw [very thick,dred, ->] (-1,-2) to(-1,-1.7) to [out=90,in=225] (0,-.5)  (0,.5) to [out=45,in=270] (1,1.7) to (1,2);
  \end{tikzpicture}
  }
  \to
    q^{-1}\;
  \begin{tikzpicture}[anchorbase,scale=.2]
    \draw [very thick, green, ->] (1,-2) to (1,0) (-1,0) to (-1,2);
    \draw [very thick,dred, ->] (-1,-2) to (-1,0) (1,0) to (1,2);
  \end{tikzpicture}
  \quad \cong \quad
  \begin{tikzpicture}[anchorbase,scale=.2]
    \draw [very thick,green, ->] (1,-2)to(1,-1.7) to [out=90,in=315] (0,-.5)(0,.5) to [out=135,in=270] (-1,1.7) to (-1,2);
    \draw[double] (0,.5) to (0,-.5);
    \draw [very thick,dred, ->] (-1,-2) to(-1,-1.7) to [out=90,in=225] (0,-.5)  (0,.5) to [out=45,in=270] (1,1.7) to (1,2);
  \end{tikzpicture}
  \quad \cong \quad 
  q\;
  \begin{tikzpicture}[anchorbase,scale=.2]
    \draw [very thick, green, ->] (1,-2) to (1,0) (-1,0) to (-1,2);
    \draw [very thick,dred, ->] (-1,-2) to (-1,0) (1,0) to (1,2);
  \end{tikzpicture}
  \to
  \uwave{
  \begin{tikzpicture}[anchorbase,scale=.2]
    \draw [very thick,green, ->] (1,-2)to(1,-1.7) to [out=90,in=315] (0,-.5)(0,.5) to [out=135,in=270] (-1,1.7) to (-1,2);
    \draw[double] (0,.5) to (0,-.5);
    \draw [very thick,dred, ->] (-1,-2) to(-1,-1.7) to [out=90,in=225] (0,-.5)  (0,.5) to [out=45,in=270] (1,1.7) to (1,2);
  \end{tikzpicture}
  }
  \quad = \quad
  \CC{\begin{tikzpicture}[anchorbase,scale=.2]
    \draw [very thick,dred, ->] (-1,-2) to (-1,-1.7) to [out=90,in=270] (1,1.7) to (1,2);
    \draw [white, line width=.15cm] (1,-2)to(1,-1.7) to [out=90,in=270] (-1,1.7) to (-1,2);
    \draw [very thick,green, ->] (1,-2)to(1,-1.7) to [out=90,in=270] (-1,1.7) to (-1,2);
  \end{tikzpicture}
  }_\fr
\]
Here the two isomorphisms are simply the projections/inclusions of the term in
homological degree zero, since the other webs are isomorphic to zero due to the
incompatible idempotent coloring.

\revise{Analogously, we define a canonical class $\immballn$ in the colored deformed
homology of the negative Hopf using the link diagram obtained by reversing signs of crossings in the above diagram.
After bending, this corresponds to the inverse of the chain map of the complexes corresponding to $\immballp$.
This implies that $\immballp$ and $\immballn$ cancel in situations we will describe below.}

We will interpret the classes $\immballpn$
as corresponding to a pair of idempotent-colored oriented disks that
transversely intersect in a single point with sign $\pm$. 

\end{example}

\begin{remark} We invite the reader to compare the computation in
\eqref{eqn:deformedHopf} with its analog for Khovanov homology, computed using
the Karoubi-envelope technology from \cite{MR2253455}. In the $\sll_2$-case of
Khovanov homology, the summands of Lee homology can be naturally parametrized by
orientations of the link. In deformations of $\gl_2$ homology, the summands are
naturally parametrized by labellings of components by roots $\lambda_i$ of the
deformation polynomial. The parametrization for the $\gl_2$ case generalizes to $\glN$
and non-simple deformations.
\end{remark}

\begin{construction}[Chromatography]\label{constr:chromatography} Let $L \subset S^3$ be a link and $c\in
C(L,\defset)$ a coloring of its components. 

We aim to construct an isomorphism
\begin{equation}\label{eqn:chromatography}
\bigotimes_{i=1}^n \skeins(B^4; c(L_{c^{-1}(\lambda_i)}))   
\cong \skeins(B^4;c(L))
\end{equation}  
and throughout the construction we will make liberal use of the isomorphisms from \eqref{eqn:B4id} and
 \eqref{eqn:B4idc}. In particular, we see that $\revise{\skeins(B^4;
 c(L_{c^{-1}(\lambda_i)}))}$ is spanned by \emph{tautological lasagna fillings}
 consisting of a single concentric input ball, decorated by the link
 $L_{c^{-1}(\lambda_i)}$ and labelled by a class $v_i\in
 \KhR_\defset(c(L_{c^{-1}(\lambda_i)}))$. Here we have written
 $c(L_{c^{-1}(\lambda_i)})$ for the sublink $L_{c^{-1}(\lambda_i)}$ with the
 coloring inherited from $c$, namely the one homogeneous of color $\lambda_i$.
 We denote the tautological lasagna filling corresponding to $v_i$ by
 $F_i(v_i)$.

Let $F_1(v_1)\otimes \cdots \otimes F_n(v_n)$ be an elementary tensor of
tautological lasagna fillings \revise{on the left-hand side of
\eqref{eqn:chromatography}. Of course, such elementary tensors span the left-hand side.} Consider the (topological) union $\cup_{i} F_i(v_i)$
of the lasagna fillings $F_i(v_i)$. This is, in general, not a lasagna filling
on its own, as the input balls $B_i$ and surfaces $\surf_i$ of different
$F_i(v_i)$ may intersect in the union. Note, however, that $\cup_{i} F_i(v_i)$ has
boundary $c(L)$ and the surface is radial near the boundary. To turn $\cup_{i}
F_i(v_i)$ into a lasagna filling, we change each $F_i(v_i)$ by an isotopy, so that
the following conditions are satisfied for all $i\neq j$:
\begin{itemize}
\item the input ball $B_i$ of $F_i(v_i)$ is disjoint from the input ball $B_j$
and surface $\surf_j$ of $F_j(v_j)$,
\item the surfaces $\surf_i$ of $F_i(v_i)$ and $\surf_j$ of $F_j(v_j)$
intersect transversely in a finite number of points. 
\end{itemize} 
Once these conditions are satisfied, we call the collection $\{F_1(v_1),\dots,
F_n(v_n)\}$ \emph{transverse}. In such a collection of fillings we have a finite
number of intersection points of surfaces, at each of which the two surfaces are
colored by different idempotents, say $e(\lambda_i)$ and $e(\lambda_j)$. We
resolve such intersections by drilling out a neighborhood of the intersection
point, which creates a new input ball decorated with an
$\lambda_i$-$\lambda_j$-colored positive or negative Hopf link, which we label
by the canonical class $\immballp$ or $\immballn$ from Example~\ref{example:Hopfsplit},
depending on the sign of the intersection. We denote the resulting lasagna
filling by $F_{\cup i}(v_1 \otimes \cdots \otimes v_n)$; it represents an
element of $\skeins(B^4;c(L))$. The assignment
\[
  F_1(v_1)\otimes \cdots \otimes F_n(v_n) \mapsto F_{\cup i}(v_1 \otimes \cdots \otimes v_n)
\] 
is our candidate for the desired isomorphism \eqref{eqn:chromatography}.
\end{construction}

\begin{proposition}\label{prop:chromatography} The \emph{chromatography} map from
\eqref{eqn:chromatography} is well-defined, an isomorphism, and natural with
respect to link cobordism in the following sense. Let $S\subset S^3 \times I$ be
a link cobordism $L \to L'$, $c\colon \pi_0(S)\to \defset$ a coloring of the
components of $S$, and also denote by $c\in C(L,\defset)$, $c\in C(L',\defset)$ the induced
colorings of $L$ and $L'$ respectively. Write 
\begin{align*}
  \skeins(c(S))\colon \skeins(B^4;c(L)) &\to \skeins(B^4;c(L'))\\
  \skeins(c(S_{c^{-1}(\lambda_i)}))\colon \skeins(B^4;c(L_{c^{-1}(\lambda_i)})) &\to \skeins(B^4;c(L'_{{c}^{-1}(\lambda_i)})) 
\end{align*}
for the linear maps induced by interpreting the idempotent colored cobordism $S$
(or one of its color-homogeneous components, in the second line) as a lasagna
filling of $S^3\times I$ and gluing it onto lasagna fillings of $4$-balls. Then
we have the following commutative naturality square for the chromatography maps:
\begin{equation}
  \label{eqn:chromatographynaturality}
			\begin{tikzcd}
				\bigotimes_{i=1}^n \skeins(B^4; c(L_{c^{-1}(\lambda_i)})) 
        \arrow[d,"\bigotimes_{i=1}^n \skeins(c(S_{c^{-1}(\lambda_i)}))"]  
        \arrow[r,"\cong"] & \skeins(B^4;c(L))
        \arrow[d,"\skeins(c(S))"]
         \\
        \bigotimes_{i=1}^n  \skeins(B^4;c(L'_{{c}^{-1}(\lambda_i)}))    
        \arrow[r,"\cong"] & \skeins(B^4;c(L'))\\
      \end{tikzcd}
    \end{equation}

\end{proposition}
Before giving the proof, we sketch an example of $F_{\cup i}(v_1 \otimes \cdots
\otimes v_n)$ and another lasagna filling $F_{\cup i}$ that will play an
auxiliary role in the proof.

\[
  F_{\cup i}(v_1 \otimes v_2 \otimes v_3) \;=\;
  \begin{tikzpicture}[anchorbase,scale=.2]
    \bsphere{0}{2}{14}
    \sphere{-4.3}{7.7}{1}
    \node at (-3.7,6) {\tiny$\immballp$};
    \sphere{0}{0}{3}
    \sphere{7}{0}{3}
    \sphere{-7}{0}{3}
    \trefoil{.5}{1}{1}{green,very thick}
    \node at (0.4,-.7) {\tiny$v_2$};
    \unknot{-7}{1}{.75}{red,very thick}
    \node at (-7.5,-.7) {\tiny$v_1$};
    \unknot{6.75}{1}{.75}{blue,very thick}
    \unknot{8.75}{1}{.75}{blue,very thick}
    \node at (7.25,-.7) {\tiny$v_3$};
    \draw[red, thick] (-8.3,.85) to [out=80,in=290] (-7.3,11)
    (-7,.85) to [out=80,in=290] (-4.4,11);
    \draw[blue, thick]
    (5.4,.85) to [out=100,in=270] (-5.7,11)
    (6.9,.85) to [out=100,in=270] (-2.7,11)
    (7.4,.85) to [out=90,in=270] (7.4,9)
    (8.8,.85) to [out=90,in=270] (10.3,9);
    \draw[green, thick]
    (0,7.4) to [out=80,in=260] (.7,11)
    (2.5,6.1) to [out=80, in=260] (5.2,10.7)
    (-1,0.2) to [out=80,in=260] (-.3,5.3)
    (2,0.2) to [out=80, in=260] (2,4.2)
    ;
    \trefoil{3}{9}{1.5}{green,very thick}
    \hopflinkcol{-5}{11}{1.5}{blue,very thick}{red,very thick}
    \unknot{9}{9}{1.5}{blue,very thick}
    \fsphere{0}{2}{14}
    \end{tikzpicture}\quad , \quad F_{\cup i} \;=\;
    \begin{tikzpicture}[anchorbase,scale=.2]
      \bsphere{0}{2}{14}
      \sphere{-4.3}{7.7}{1}
      \node at (-3.7,6) {\tiny$\immballp$};
\draw[white, line width=1mm] (-10,0) \dr (-7,-2) to (7,-2) \ru (10,0);
\draw  (-10,0) \dr (-7,-2) to (7,-2) \ru (10,0);
\draw[thick] 
(-10,0) \dr (-7,-3) to (7,-3) \ru (10,0) 
(-10,0) \ur (-7,3)to (7,3) \rd (10,0);
\draw[dashed] (-10,0) \ur (-7,2)to (7,2) \rd (10,0);
      \trefoil{.5}{1}{1}{green,very thick}
      \unknot{-7}{1}{.75}{red,very thick}
      \unknot{6.75}{1}{.75}{blue,very thick}
      \unknot{8.75}{1}{.75}{blue,very thick}
      \draw[red, thick] (-8.3,.85) to [out=80,in=290] (-7.3,11)
      (-7,.85) to [out=80,in=290] (-4.4,11);
      \draw[blue, thick]
      (5.4,.85) to [out=100,in=270] (-5.7,11)
      (6.9,.85) to [out=100,in=270] (-2.7,11)
      (7.4,.85) to [out=90,in=270] (7.4,9)
      (8.8,.85) to [out=90,in=270] (10.3,9);
      \draw[green, thick]
      (0,7.4) to [out=80,in=260] (.7,11)
      (2.5,6.1) to [out=80, in=260] (5.2,10.7)
      (-1,0.2) to [out=80,in=260] (-.3,5.3)
      (2,0.2) to [out=80, in=260] (2,4.2)
      ;
      \trefoil{3}{9}{1.5}{green,very thick}
      \hopflinkcol{-5}{11}{1.5}{blue,very thick}{red,very thick}
      \unknot{9}{9}{1.5}{blue,very thick}
      \fsphere{0}{2}{14}
      \end{tikzpicture}  
\]

\begin{proof}
  First we check that the chromatography map is well-defined. For given lasagna
  fillings $F_1(v_1)\otimes \cdots \otimes F_n(v_n)$, the filling $F_{\cup
  i}(v_1 \otimes \cdots \otimes v_n)$ depends on the choice of a small isotopy
  to make the original fillings transverse. Different isotopies result in unions
  $F_{\cup i}(v_1 \otimes \cdots \otimes v_n)$ which differ by isotopy and the
  addition or removal of pairs of $\immballp$- and $\immballn$-labeled intersections \revise{which are connected by a Whitney disk. These}
  cancel in $\skeins(B^4;c(L))$, see Example~\ref{example:Hopfsplit}. The union is thus well-defined. 

  To show that the chromatography map is an isomorphism and natural, it \revise{is} helpful to give a
  skein-theoretic description. To this end, we remove the labels
  $v_1,\cdots,v_n$ from the input balls in $F_{\cup i}(v_1 \otimes \cdots
  \otimes v_n)$ and view what remains as a lasagna diagram (see \cite[Definition
  5.1]{MR4562565}) with $\immballpn$ insertions and $n$ input balls
  decorated with the $c(L_{c^{-1}(\lambda_i)})$ for $1\leq i \leq n$. As in the
  proof of \cite[Theorem 5.2]{MR4562565}, we can tube the input balls
  together to obtain a lasagna filling $F_{\cup i}$ of $S^3\times I$ with
  colored link $c(L)\subset S^3 \times \{1\}$ on one side and the
  color-separated link $\bigsqcup_{i=1}^n c(L_{c^{-1}(\lambda_i)})\subset S^3
  \times \{0\}$ on the other side. 
  
  Note that the chromatography map \eqref{eqn:chromatography} can be recovered
  from $F_{\cup i}$, by considering the linear map \[\skeins\left(B^4;\bigsqcup
  _{i=1}^n c(L_{c^{-1}(\lambda_i)})\right)  \to \skeins(B^4; c(L))\] induced by
  gluing on $F_{\cup i}$ and using the monoidality of the $4$-ball skeins with
  respect to boundary connect sum, which is strict (not lax) since we work
  over $\C$.

 Analogously, we can construct another lasagna filling $\overline{F_{\cup i}}$
  of $S^3\times I$ with the color separated link on the $S^3 \times \{0\}$-side
  and $c(L)$ on the $S^3 \times \{1\}$-side, simply by reversing $F_{\cup i}$
  and adjusting orientations. The invertibility of the chromatography map
  \eqref{eqn:chromatography} then follows by showing that the composed lasagna
  fillings $F_{\cup i}\circ \overline{F_{\cup i}}$ and $\overline{F_{\cup i}}
  \circ F_{\cup i}$ are equivalent to identity link cobordisms in the skein of
  $S^3\times I$, since the $\immballpn$ insertions in $F_{\cup i}$ cancel with
  $\immballnp$ insertions in $\overline{F_{\cup i}}$ in pairs. 
  
The naturality squares \eqref{eqn:chromatographynaturality} also follow from the
skein-theoretic description. By the same arguments as used for the
well-definedness above, the result of gluing the lasagna filling corresponding
to the idempotent-colored cobordism $c(S)$ to the output side \revise{$S^3\times \{1\}$} of $F_{\cup i}$ \revise{in $S^3\times I$}
is a lasagna filling that is equivalent to the one obtained by gluing the
filling $\bigsqcup_{i=1}^n c(S_{c^{-1}(\lambda_i)})$, corresponding to the split
disjoint union of the colored components of $S$, to the input side \revise{$S^3\times \{0\}$ of $F_{\cup i}$}. 
\end{proof}

\begin{remark}
Proposition~\ref{prop:chromatography} can be interpreted as saying that the
disk-like $4$-category controlling $\skeins$ is the tensor product of its
monochromatic sub-$4$-categories. The generating $2$-morphism in $\skeins$,
which is the default coloring of all links and link cobordisms, corresponds to
the direct sum of the generating $2$-morphisms in the monochromatic versions.
\end{remark}

\begin{proof}[Proof of Theorem~\ref{thm:defskein}] We will use the direct sum
decomposition from Lemma~\ref{lem:onecolattime} and hence focus on a single
coloring $c\in C(L,\defset)$. We first prove a variation of the desired
isomorphism \eqref{eqn:chromatographyskein}, namely:
\begin{equation}
\label{eqn:chromatographyskeinpercol}
\skeins(W;c(L)) \xleftarrow{\cong} \bigotimes_{i=1}^n \skeinsmon(W; L_{c^{-1}(\lambda_i)})   
\end{equation}
where $\skeinsmon(W; L_{c^{-1}(\lambda_i)})$ denotes the monochromatic subspace
of $\skeins(W; L_{c^{-1}(\lambda_i)})$ in which every lasagna filling has its
surface components colored by the idempotent $e(\lambda_i)$. The isomorphism can
be constructed as in the preceding discussion for $W=B^4$: Given an elementary
tensor of lasagna fillings representing an element on the right-hand side of
\eqref{eqn:chromatographyskeinpercol}, we form a superposition of lasagna
fillings and resolve transverse intersections using $\immballpn$ classes
following Construction~\ref{constr:chromatography} and show, as in the proof of
Proposition~\ref{prop:chromatography}, that this gives a well-defined class of
lasagna filling on the left-hand side. The map thus defined is surjective since
every lasagna filling $F$ representing an element on the left-hand side can be
expressed as a linear combination of superposed fillings. To see this, we expand
all components of the surface of $F$ into linear combinations of their
idempotent-colored components and split all input balls of $F$ into a
superposition of monochromatic input balls using \eqref{eqn:chromatography}.
Conversely, \revise{the candidate inverse map, which sends any superposed filling to the tensor product of its
monochromatic components (after removing the $\immballpn$ labels), is indeed well-defined by Proposition~\ref{prop:chromatography}, since it respects isotopies and local
skein relations, and thus is an inverse to the superposition map defined above.}

Finally we claim that there is an\footnote{The naturality of this isomorphism will be discussed in the following.} isomorphism of $H_2(W\revise{,} L)\times \Z_t$-graded
$\C$-vector spaces:
\[
  \skeinsmon(W; L_{c^{-1}(\lambda_i)})   \cong \skeini(W; L_{c^{-1}(\lambda_i)}, \C)  
\]
As before, it suffices to prove this in the case $W=B^4$, where it reduces to 
\[\KhR_{\defset,\fr}(c(L_{c^{-1}(\lambda_i)})) \cong
\KhR_{N_i,\fr}(L_{c^{-1}(\lambda_i)}, \C)   \] and the existence of such an
isomorphism is proven as a part of Theorem~\ref{thm:deflinks} in
\cite{MR3590355}. However, here we need isomorphisms that are \emph{natural}
under link cobordisms. The proof in \cite{MR3590355} uses that the operation of
coloring all foam sheets with idempotents related to $\lambda_i$ and translating
dots $\bullet \mapsto \bullet-\lambda_i$ constitutes an equivalence between
$\gl(N_i)$-foams and the $\lambda_i$-colored monochromatic subcategory of the
$\defset$-deformed foams. These checks are done for foams constructed via
categorified quantum groups and hence do not involve caps, cups, and saddles.
When taking these additional generators into account, the equivalence requires a
scaling-correction depending on the indices of the involved Morse critical
points. For example, in the $\gl_{N_i}$ theory, an $(N_i-1)$-dotted sphere
evaluates to $1$, whereas a $e(\lambda_i)(\bullet-\lambda_i)^{N_i-1}$-decorated
sphere evaluates to $\revise{\epsilon_{\Sigma,i,j}:=}\prod_{j\neq i}(\lambda_i-\lambda_j)^{-N_j}$ in the
$\defset$-deformed theory. Upon proceeding from the foam theory to the link
homology theory, we see that the theories produce isomorphic vector spaces for
links, but that the \revise{maps associated to link cobordisms $S$ acquire an Euler characteristic-dependent rescaling by $\epsilon_{\Sigma,i,j}^{\chi(S)/2}\in \C$ (after choosing a square root) relative to the $\gl_{N_i}$ theory}. We thus get a natural
isomorphism of link homologies
\[\KhR_{\defset,\fr}(c(L_{c^{-1}(\lambda_i)})) \cong
\KhR^{\mathrm{tw}}_{N_i,\fr}(L_{c^{-1}(\lambda_i)}, \C)   \] where we have
placed $\mathrm{tw}$ as a warning flag to indicate the twisting. Thus we have
established the natural isomorphism \eqref{eq:ThmC} from Theorem C in the
introduction. This now implies an isomorphism of skein modules
\[
  \skeinsmon(W; L_{c^{-1}(\lambda_i)})   \cong \skeinitw(W; L_{c^{-1}(\lambda_i)}, \C)  
\]
and thus \eqref{eqn:chromatographyskein}.
\end{proof}

\subsection{Relative second homology as link homology}

For the following we define a very simple link homology theory\footnote{Since it
is not computed from chain complexes associated to link diagrams, it might be
more accurately called a \emph{homology link theory}.} based on relative second
homology. For a link $L\subset S^3$ of self-linking $f\in \Z$ we denote by 
\[H(L):=\Z^{H_2(B^4)^L}\] the $\Z_q\times \Z_t$-graded rank one free abelian
group $\Z$, concentrated in $\Z_q\times \Z_t$-degree $(-f,f)$, which is
generated by the unique relative second homology class in $H_2(B^4)^L$, namely
the preimage of $[L]$ under the connecting homomorphism. For a link cobordism
$\surf\colon L_0 \to L_1$ in $S^3 \times I$, we associate a grading-preserving
homomorphism between the respective groups as: pick an oriented (unframed!)
smooth surface $S'$ representing the unique class in $H_2(B^4)^{L_0}$ and thus
the canonical generator $[S']$ of $H(L_0)$. Then we define 
\begin{align*}
  H(\surf) \colon H(L_0) &\to H(L_1)\\
[S'] &\mapsto [\surf \circ S']
\end{align*}
This construction \revise{is independent of the choice of $S'$ and manifestly functorial, thus defining} a $\Z_q\times
\Z_t$-graded link homology theory for links in $S^3$. \revise{Note, however, that $H(\surf)$ is also independent of the choice of link cobordism $\surf$.}

\begin{proposition}\label{prop:glone} There exists a natural isomorphism of $\Z_q\times
  \Z_t$-graded link homology theories:
  \[\KhR_{1,\fr}(-) \cong H(-)\]
\end{proposition}
\begin{proof}
  In the $\glone$ homology theory, the crossing formulas simplify to:
  \begin{equation}
    \CC{\begin{tikzpicture}[anchorbase,scale=.2]
      \draw [very thick, ->] (1,-2)to(1,-1.7) to [out=90,in=270] (-1,1.7) to (-1,2);
      \draw [white, line width=.15cm] (-1,-2) to (-1,-1.7) to [out=90,in=270] (1,1.7) to (1,2);
      \draw [very thick, ->] (-1,-2) to (-1,-1.7) to [out=90,in=270] (1,1.7) to (1,2);
    \end{tikzpicture}
    }_\fr
     \quad = \quad
     0 \to
      q^{-1} t\;
    \begin{tikzpicture}[anchorbase,scale=.2]
      \draw [very thick, ->] (1,-2) to (1,2);
      \draw [very thick, ->] (-1,-2) to (-1,2);
    \end{tikzpicture}
    \to 0
    , \qquad
    \CC{\begin{tikzpicture}[anchorbase,scale=.2]
      \draw [very thick, ->] (-1,-2) to (-1,-1.7) to [out=90,in=270] (1,1.7) to (1,2);
      \draw [white, line width=.15cm] (1,-2)to(1,-1.7) to [out=90,in=270] (-1,1.7) to (-1,2);
      \draw [very thick, ->] (1,-2)to(1,-1.7) to [out=90,in=270] (-1,1.7) to (-1,2);
    \end{tikzpicture}
    }_\fr
     \quad = \quad 
     0 \to
     q t^{-1}\;
    \begin{tikzpicture}[anchorbase,scale=.2]
      \draw [very thick, ->] (1,-2) to (1,2);
      \draw [very thick, ->] (-1,-2) to (-1,2);
    \end{tikzpicture}
    \to 0
     \end{equation}
And the $\glone$ foam category underlying the $\glone$ homology has no dots,
seams, or vertices. The sphere and neck-cutting relations take the following simple form:
\begin{equation}
  \label{eqn:glone}
  \raisebox{-0.5\height}{\includegraphics{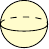}}
  =1 
  \;\; , \;\;
  \raisebox{-0.5\height}{\includegraphics{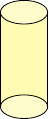}} 
  \; = \;
  \raisebox{-0.5\height}{\includegraphics{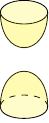}}
\end{equation}
Moreover, all cobordisms preserve the $\Z_q\times \Z_t$ bidegree. From this
description, it is clear that the $\glone$ homology of any link diagram is free of
rank one, concentrated in bidegree $(-w,w)\in \Z_q\times \Z_t$, where $w\in \Z$
denotes the writhe of the diagram: the difference of positive and negative
crossings. For a blackboard-framed diagram, the writhe computes the self-linking
number $w=f$.

As for cobordism maps, the Reidemeister II and III moves as well as the framed
Reidemeister I moves are implemented by chain maps given by identity foams,
concentrated in a single degree. The Morse generators for cobordisms are simply
given by composition by the corresponding surfaces, interpreted as foams. By the
relations \eqref{eqn:glone}, cobordism maps in the link homology theory
$\KhR_{1,\fr}(-)$ only depend on the relative second homology class represented
by the cobordism\revise{, and this is determined by the source and target links}.

Let $v_{\circ}$ denote the canonical class in the $\glone$ homology of the
$f$-framed unknot. (It can be represented by $|f|$ insertions of framing-changing
input balls $\frameballpn$ according to the sign of $f$ and it is clearly of
minimal $q$-degree.) For any link $L$ of self-intersection $f$, consider a
connected Seifert surface $\surf^L$, remove a disk and interpreted the result
$\surf^L_\circ$ as a link cobordism $\surf$ from the $f$-framed unknot to $L$.
By the discussion above, we obtain a canonical generator
$v_{L}:=\KhR_{1,\fr}(\surf^L_\circ)(v_{\circ})\in \KhR_{1,\fr}(L)$ that is
independent of the choice of $\surf$. The claimed natural isomorphism of link
homologies is now given by 

\begin{align*}
  \KhR_{1,\fr}(L) &\xrightarrow{\cong} H(L)\\
  c\cdot v_L &\mapsto c\cdot[\surf^L]
\end{align*}
for $c\in \Z$.
\end{proof}

By virtue of Proposition~\ref{prop:glone}, the link homology theory $H(-)$ also
satisfies the requirements of Theorem~\ref{thm:skeinfromLH} and hence gives rise
to skein modules for $4$-manifolds $W$ and links $L\subset \partial W$. In fact,
by the natural isomorphism, these skein modules are isomorphic to
$\skeinone(W;L)$. In the following theorem, we will obtain an alternative
characterization, purely in terms of second homology. To this end, we define
$H(W,L)$ to be the $H_2(W)^L\times \Z_q\times \Z_t$-graded free abelian group
freely generated by the relative second homology classes bounding $L$, i.e. by
$H_2(W)^L$, with classes $\alpha\in H_2(W)^L$ considered homogeneous of
tridegree $(\alpha, -\alpha\cdot \alpha, \alpha \cdot \alpha)$.

\begin{theorem}\label{thm:oneskein} Let $W$ be a $4$-manifold and link $L\subset \partial W$ a link
  in the boundary. Then we have an isomorphism of $H_2(W\revise{,} L)\times
  \Z_q \times \Z_t$-graded abelian groups:
  \[\skeinone(W;L)\cong H(W,L)\]
\end{theorem}
\begin{proof}
Let $\surf$ be a surface representing a given homology class $\alpha \in
H_2(W)^L$.
As discussed in Section~\ref{sec:skeinsfromsurfaces}, this gives rise to a class
$[\surf]\in \skeinone(W;L)$ of tridegree $(\alpha, -\alpha\cdot \alpha, \alpha
\cdot \alpha)$. Any other representative $\surf'$ of $\alpha$ differs by an
isotopy, a finite sequence of neck-cutting moves and creations and deletions of
$S^2$-components. In the $\glone$ foam category and thus in the skein module
$\skeinone(W;L)$, these moves are among the skein relations and hence imply
$[\surf]=[\surf']$. Thus we have a well-defined assignment $[\surf] \mapsfrom
\alpha$, which linearly extends to a morphism of $H_2(W\revise{,} L)\times \Z_q \times
\Z_t$-graded abelian groups:
\[
\skeinone(W;L) \leftarrow H(W,L)
\]
Conversely, the proof of Proposition~\ref{prop:glone} shows that every lasagna filling
$F$ in $\skeinone(W;L)$ can be represented by a multiple, with scalar denoted
$c(F)$, of a surface $\surf(F)$ with framing-changing input balls inserted
(but no other input balls). The surface $\surf(F)$ may be chosen
freely in its relative second homology class $\alpha \in H_2(W\revise{,} L)$ (by skein
relations) and the total framing correction is determined by $\alpha\cdot
\alpha$. The assignment $F\mapsto c(F) \alpha$ is then a 2-sided inverse to \revise{the} map constructed above.
\end{proof}

  Let $d\in \C^\times$. In complete analogy to the discussion
in Proposition~\ref{prop:glone}, one can define a \emph{twisted} version of $\glone$ link
homology based on the skein relations:
  \begin{equation}
    \raisebox{-0.5\height}{\includegraphics{thinsphere}}
    =1/d 
    \;\; , \;\;
    \raisebox{-0.5\height}{\includegraphics{thincylinder}} 
    \; = \; 
    d 
    \raisebox{-0.5\height}{\includegraphics{capcup}}
  \end{equation}
  This results in a theory that is isomorphic to the $\glone$ theory, although
  not naturally, since cobordism maps have to be rescaled by a power of $d$. In fact, for an element $\lambda_i\subset
  \defset$ of multiplicity $N_i=1$, the $\lambda_i$-monochromatic subtheory of
  the $\defset$-deformed link homology is isomorphic to the twisted
  $\glone$-theory with parameter $d=\prod_{j\neq i}(\lambda_i-\lambda_j)^{N_j}$,
  see the proof of Theorem~\ref{thm:defskein}. We will denote the associated
  link homology theory based on relative second homology, but defined over $\C$,
  by $H^\mathrm{tw}(L)$ if the scalar $d$ is clear from the context or
  irrelevant. Similarly, we will write $H^\mathrm{tw}(W\revise{,}L,\C)$ for the
  associated skein modules.

\begin{corollary} \label{cor:completedeformation} Consider a set of deformation
parameters $\defset:=\{\lambda_1,\dots,\lambda_N\}\subset \C$ with multiplicities
$N_i=1$. Then we have an isomorphism $H_2(W\revise{,} L)\times
\Z_q \times \Z_t$-graded $\C$-vector spaces:
\begin{equation}
  \label{eqn:chromatographyGornikskein}
\skeins(W;L) \cong \bigoplus_{c\in C(L,\defset)} \bigotimes_{i=1}^N H^{\mathrm{tw}}(W, L_{c^{-1}(\lambda_i)},\C)   
\end{equation}
Moreover, given an oriented surface $S\subset W$ with boundary $L$ and no closed
components, this defines an element of the left-hand side (see
Definition~\ref{def:skeinfromsurface}), whose corresponding $c$-summand on the
right hand side is given by the tensor product $[S_{c^{-1}(\lambda_1)}]\otimes
\cdot \otimes [S_{c^{-1}(\lambda_N)}]$ of the relative second homology classes
represented by the monochromatic components of $S$ with respect to the coloring
induced by $c$ from the boundary.
  \end{corollary}
  \begin{proof}
    This follows from \eqref{eqn:chromatographyskein} and the identification of
    (twisted) $\glone$ theories in terms of (twisted) relative second homology. 
  \end{proof}

\begin{corollary} \label{cor:nonvanishingdef} Consider a set of deformation
  parameters $\defset:=\{\lambda_1,\dots,\lambda_N\}\subset \C$ with
  multiplicities $N_i=1$. Given an oriented surface $S\subset W$ with boundary
  $L$ and no closed components, then the associated lasagna filling in $\skeins(W;L)$ is nonzero.
\end{corollary}
\begin{proof} Immediate from Corollary~\ref{cor:completedeformation}.
\end{proof}

In the following we will see that the assumption in
Corollary~\ref{cor:nonvanishingdef} on $S$ having no closed components can be
weakened. Note, however, that we do have to exclude certain closed surfaces
contained in $4$-balls inside of $W$, which would already evaluate to zero
locally. 

\begin{proposition}
  \label{prop:techcond}
Retain the assumptions on $\surf$ from Corollary~\ref{cor:nonvanishingdef}.
For an oriented surface $S\subset W$ the associated lasagna filling in
$\skeins(W;L)$ is nonzero provided the following assumption is satisfied:
\begin{center}
  (Homological diversity) No union of a \revise{nonempty} subset of closed connected
  components of $S$ is nullhomologous in $W$.
\end{center}
\end{proposition}
\begin{proof}
\revise{Fix a 
color $\lambda\in\Sigma$ and let
$c \in C(L,\Sigma)$ be the corresponding ``monochromatic" (constant) coloring.
The $c$-summands of the}
right-hand sides of \eqref{eqn:chromatographyskein} and 
  \eqref{eqn:chromatographyGornikskein} are graded by $\revadd{\prod_{i=1}^N} H_2(W,
  L_{c^{-1}(\lambda_i)})$. The lasagna filling determined by $S$ is the sum over
  the fillings obtained by coloring \emph{all} connected components by
  idempotents $e(\lambda_i)$ (not just the non-closed components). Each of these
  summands is nonzero by the same argument as in
  Corollary~\ref{cor:completedeformation}. The additional assumption of
  homological diversity guarantees that 
  \revise{the monochromatic $c$-colored contribution to the sum is the only one living in its homogeneous component with respect to the $\prod_{i=1}^N H_2(W,
  L_{c^{-1}(\lambda_i)})$-grading}
  and therefore cannot be cancelled by other terms in the sum.
\end{proof}

\revise{Homological diversity is a convenient sufficient condition for non-vanishing, but it is straightforward to envision alternative conditions, where mild variations of this proof apply.}

\subsection{Non-vanishing modulo torsion}
\label{sec:nonvanishing} 
Let $\defset=\{\lambda_1^{N_1},\dots \lambda_n^{N_n}\}$ be a multiset of complex
  deformation parameters. Recall the bimodule $R_\defset$ from \eqref{eqn:RSigma}, which
  is used to relate the equivariant and $\defset$-deformed unknot homologies.

\begin{lemma}\label{lem:equivtodef}  There exists a natural transformation of link homology theories 
  \[
    R_\defset\otimes \KhR_{\GLN,\fr}(-) \Rightarrow \KhR_{\defset,\fr}(-)   
  \]
\end{lemma}
\begin{proof}
  The equivariant theory $\KhR_{\GLN,\fr}(-)$ is computed as the homology of
  chain complexes of free $R_\GLN$-modules associated to link diagrams and chain
  maps associated to movies of link cobordisms. The $\defset$-deformed theory
  arises as the homology of the same chain complexes and chain maps, with
  coefficients reduced to $\C$ by tensoring with the bimodule $R_\defset$, i.e.
  by specializing the equivariant parameters. The universal coefficient theorem
  thus yields the desired natural transformation.
\end{proof}

\begin{proposition}\label{prop:equivtodef} There exists a $H_2(W\revise{,} L)\times
\Z_t$-grading preserving linear map
  \begin{equation}
    \label{eqn:equivtodef}
    R_\defset\otimes \skeine(W;L) \to \skeins(W;L)
  \end{equation}
  which sends classes associated to oriented surfaces $S\subset W$ with boundary
  $L$ in $\skeine(W;L)$ to the analogous classes in $\skeins(W;L)$.
\end{proposition}
\begin{proof}
  The map is defined on lasagna fillings by applying the components of the
  natural transformation from \eqref{lem:equivtodef} to the labels of the input
  balls. It follows from the naturality of the transformation that this
  assignment respects skein relations and hence descends to a well-defined
  linear map on the level of skein modules.
\end{proof}

\begin{corollary}
Any homologically diverse oriented surface $S\subset W$ with boundary $L$
defines an class $[S]$ in the $R_\GLN$-module $\skeine(W;L)$ that is non-trivial
modulo torsion.
\end{corollary}
\begin{proof}
 If $[S]$ were torsion, we could find a set $\defset$ of complex deformation
 parameters, such that the image of $[S]$ under the map from
 \eqref{eqn:equivtodef} vanishes, in contradiction to
 Proposition~\ref{prop:techcond}.
\end{proof}

This finishes the proof of Theorem A and Corollary B is an immediate consequence.

\subsection{\revadd{Computations}}
\label{sec:comp}


For $W=B^4$, we obtain a generalization of Rasmussen's \cite[Theorem 4]{MR2729272}.

\begin{lemma}\label{lem:poslink}
Let $L$ be a positive \revise {framed} link \revise{in $S^3$}, which we take to be represented by a \revise{blackboard-framed} link diagram \revise{$D$}
containing only positive crossings, \revise{and let $\alpha$ denote the unique homology class in $H_2(B^4)^L$}. Then the genus
bound resulting from $q_{\min}(B^4,L,\alpha)$ is sharp, witnessed by a Seifert
surface for $L$.
\end{lemma}
\begin{proof}
  Let $n\geq 1$ be the number of crossings in the positive diagram $D$ for $L$
  (the case of the unlink is easy to verify separately). Then for the unique
  relative second homology class $\alpha\in H_2(B^4)^L$ we have
  $\alpha\cdot\alpha=n$. Let $k$ be the number of circles in the oriented
  resolution of $D$. The chain complex computing the $\GLN$-equivariant link
  homology is a complex of free $A:=\KhR_{\GLN}(\bigcirc)$-modules and since the
  positive diagram $D$ is $+$adequate, see e.g. \cite{MR2399065}, the rightmost
  end of the complex is of the form
  \begin{equation}
    \label{eqref:poslink} \cdots \to \bigoplus_{n} q^{1-n}t^{n-1} A^{\otimes(k-2)}\otimes B  \to q^{-n}t^n A^{\otimes
k}\end{equation} 
\revise{where $A$ and $B$ are such that} $R1+:=(B \to q^{-1}t^1 A^{\otimes 2} \simeq q^{-N}t^1 A)$ \revise{describes} the complex
computing the homology of the $1$-framed unknot from its 1-crossing diagram.
The components of the single differential shown in
\eqref{eqref:poslink} are obtained from the differential in $R1+$ by tensoring
on the left or right with copies of identity morphisms on $A$ and possibly
adjusting by a sign. This implies, that the homology of \eqref{eqref:poslink} in
the right-most homological degree, the cokernel of the shown differential, is
isomorphic to the right-most homology of the complex computing the homology of
the $(k-1)$-framed unknot from its blackboard framed diagrams with $k-1$
crossings, shifted in bidegree by $((k-1)-n,n-(k-1))$. The right-most homology is
hence free over $R_\GLN$, with lowest $q$-degree generator supported in bidegree
$(k(1-N)-n,n)$.
Here the contributions to the $q$-degree are $(k-1)-n$ from the shift mentioned
above, $-N(k-1)$ for the $(k-1)$-fold framing shift relative to the $0$-framed unknot
homology and $1-N$ for the lowest generator in the $0$-framed unknot homology.

Seifert's algorithm produces a Seifert surface $S$ of Euler characteristic
$k-n$. We check the effectiveness of the genus bound from Corollary B by
computing:
\[
  \chi(S)=k-n \leq \frac{-N [S]\cdot[S] - q_{\min} }{N-1} = \frac{-N n - (k(1-N)-n) }{N-1} = k-n
\]
so the bound is sharp.
\end{proof}
\begin{remark}
  Let us compare the situation for a negative link with $n\geq 1$. In that case we have a chain complex:
  \begin{equation}
    \label{eqref:neglink} q^{n}t^{-n} A^{\otimes
    k} \to \bigoplus_{n} q^{n-1}t^{1-n} A^{\otimes(k-2)}\otimes B  \to \cdots \end{equation} 
Here we compare to the complex $R1-:=(q^{1}t^{-1} A^{\otimes 2} \to B \simeq
q^{N}t^{-1} A)$ that computes the homology of the $(-1)$-framed unknot from its
1-crossing diagram. The homology of \eqref{eqref:neglink} in the left-most
homological degree, the kernel of the shown differential, is isomorphic to the
left-most homology of the complex computing the homology of the $(1-k)$-framed
unknot from its blackboard framed diagrams with $k-1$ crossings, shifted in
bidegree by $(n-(k-1),(k-1)-n)$. The left-most homology is hence free over
$R_\GLN$, with lowest $q$-degree generator supported in bidegree
$((2-k)(1-N)+n,-n)$. The genus bound from Corollary B is:
\[
  \chi(S)=k-n \leq \frac{-N [S]\cdot[S] - q_{\min} }{N-1} = \frac{-N (-n) - ((2-k)(1-N)+n) }{N-1} = n+ (k-2)
\]
Unless $n=1$, this is not sharp.

\end{remark}

Let $W$ be a $4$-manifold that admits a handle decomposition with a single
$0$-handle and no $1$-handles, and let $L\subset \partial W$ be a link. In
\cite{MR4589588}, building on
\cite{2020arXiv200908520M}, Manolescu and two of the authors have shown that the
skein module $\skein(W;L)$ can be computed as a quotient of a countable direct
sum of $\glN$ link homologies $\KhR_N(L\cup K_\gamma)$ with suitable grading
shifts, where $K_\gamma$ ranges over certain cables of the attaching link of the
$2$-handles of $W$. This is based on a careful study of the action of
attaching/removing $4$-, $3$- and $2$-handles to $\skein(W;L)$. The same results
hold mutatis mutandis in the grading conventions of Remark~\ref{rem:gradconv}
and in the equivariant setting, i.e. for $\skeine(W;L)$. This yields a recipe
for computing $q^N_{\min}(W,L,\alpha)$ based on the homologies of a family of links and cobordism maps between them. \revadd{For $W=S^2 \times D^2$ presented as $B^4$ with a $2$-handle attached along the $0$-framed unknot, this can be done as in \cite[Section 5]{2020arXiv200908520M}.}

\begin{remark}[Approaching the Thom conjecture/Kronheimer--Mrowka theorem via skein theory]
  \label{rem:Thom}
  For $f \geq 0$ let $L^f_{n,m}$ denote the result of cabling an $f$-framed
    unknot with a parallel cable of $n$ strands running one way and $m$ strands
    running the other way. In the following we specialize to the case $f=1$ and
    $n=m+r$ for $r\geq 0$.   The $s$-invariant of $L^1_{n,m}$ was computed in \cite{ren2023lee} As
    \[
      s(L^1_{n,m})=(r-1)^2 -2m
      \]
    which yields the following bound for the Euler characteristic for a surface in $B^4$ bounding $L^1_{n,m}$:
\[
\chi(S) \leq 2m + 2 r - r^2 
\]
This bound is sharp for $m=0$. Conversely, the Thom conjecture, first proven by
Kronheimer--Mrowka using gauge theory \cite{MR1306022}, implies that the minimal
genus of a connected smooth surface $S'\subset \C P^2$ in homology class $[S']=r[\C P^1]$
for $r\geq 1$ is $(r-1)(r-2)/2$. Consider the standard handle decomposition of
$\C P^2$ with a single $2$-handle attached along the $1$-framed unknot. Suppose
that the surface $S'$ intersects the cocore of the $2$-handle transversely in $2m+r$ points, $m+r$ positively and $m$ negatively. Then
drilling out the cocore yields a connected surface $S$ of Euler characteristic
\[
\chi(S)= \chi(S')- (2m+r) \leq  2- (r-1)(r-2) -(2m+r) =  -2m +2r - r^2
\] 
which is strictly less than the bound derived from the $s$-invariant if
$m>1$. Let us assume that we work with a variant of equivariant $N=2$ skein
modules, for which the Euler characteristic bounds given by $q^2_{\min}$
coincide with the bounds from the $s$-invariant. Then we obtain:

\[
   q^2_{\min}(B^4,L^1_{n,m},0) = -2m - 2r + r^2
\]

By an equivariant analog of the 2-handle formula of Manolescu--Neithalath
\cite{2020arXiv200908520M} or, relatedly, an equivariant analog of the Kirby
color for Khovanov homology \cite{hogancamp2022kirby}, the $N=2$
skein module of $\C P^2$ is expected to admit a surjection:
\[\bigoplus_{m\geq 0}
q^{-2m-r}\skeine(B^4;L^1_{m+r,m}, 0) \twoheadrightarrow \skeine(\C P ^2;\emptyset, r [\C P^1])  \] where the term indexed
by $m$ has mod-torsion part supported in $q$-degrees bounded below by
$-4m-3r+r^2$. 

If one could show that the image of everything below degree $-3r+r^2$ is
torsion, this would imply 
\[q^2_{\min}(\C P ^2,\emptyset, r [\C P^1])\geq -3r+r^2\] and the associated
Euler characteristic bound from Corollary~B applies to the
connected (and thus homologically diverse) surface $S'$. This would amount to a purely
skein-theoretic proof of the Thom conjecture. \smallskip

We have tested this idea and preliminary computations in a simplified setting
(based on equivariant Khovanov homology over $\C[X,\alpha]/(X^2-\alpha)$)
indicate that there will be non-torsion in lower degrees, unfortunately. We have
decided to include this remark anyway, since the strategy might also be applicable in
other skein theories.

\end{remark}

\renewcommand*{\bibfont}{\small}
\setlength{\bibitemsep}{0pt}
\raggedright
\printbibliography

\end{document}